\documentclass [12pt, reqno]{amsart}

\usepackage{times}

\usepackage{amssymb,latexsym,amsmath,amsfonts, eucal}
\usepackage[usenames,dvipsnames,svgnames,table]{xcolor}

\newtheorem{thm}{Theorem}[section]
      \newtheorem{lemma}[thm]{Lemma}
      \newtheorem{prop}[thm]{Proposition}
      \newtheorem{cor}[thm]{Corollary}
      \newtheorem{defn}[thm]{Definition}
      \newtheorem{examp}[thm]{Example}
      
      \newtheorem{rmk}[thm]{Remark}
      
      \newtheorem{Hypo}[thm]{Hypothesis}
      \numberwithin{equation}{section}

\title [Vector valued de Branges spaces]{Some aspects of vector valued de Branges spaces of entire functions}
\author[ Mahapatra]{Subhankar Mahapatra$^1$}

\address{
\begin{tabular}[t]{l}
	Department of Mathematics\\
	Indian Institute of Technology Ropar\\
	Rupnagar-140001\\
	India.
	\end{tabular}
	}

\author[Sarkar]{Santanu Sarkar$^2$}

\begin{document}
	

\keywords{ de Branges operator, Fredholm operator valued entire functions, Factorizations, Isometric embedding, Associated function, Operator nodes\\
$1$ Email: subhankar.19maz0001@iitrpr.ac.in, subhankarmahapatra95@gmail.com\\
$2$ Corresponding author; Email: santanu@iitrpr.ac.in, santanu87@gmail.com.}

\begin{abstract}
\noindent This paper deals with certain aspects of the vector valued de Branges spaces of entire functions that are based on pairs of Fredholm operator valued functions. Some factorization and isometric embedding results are extended from the scalar valued theory of de Branges spaces. In particular, global factorization of Fredholm operator valued entire functions and analytic equivalence of reproducing kernels of de Branges spaces are discussed. Additionally, the operator valued entire functions associated with these de Branges spaces are studied, and a connection with the operator nodes is established.
\end{abstract}

\maketitle
\tableofcontents
\section{Introduction}
\label{Sec-1}
L. de Branges initiated the study of Hilbert spaces of scalar valued entire functions, which generalizes the renowned Paley-Wiener spaces. This theory was started with the article \cite{Branges e1} and gained recognition with the papers \cite{Branges e2}, \cite{Branges e3}, and many more. Since then, many mathematicians have contributed to this theory, and several generalizations have emerged. A comprehensive study, including numerous examples and applications of de Branges spaces of entire functions, can be found in \cite{Branges 4}. The  theory of de Branges spaces consisting of $\mathbb{C}^n$-valued entire functions has appeared greatly in the works of D. Arov and H. Dym. These spaces have played a pivotal role in their investigations of direct and inverse problems for canonical systems of differential equations and Dirac-Krein systems. An in-depth study of these de Branges spaces with matrix valued reproducing kernels in connection with the theory of $J$-contractive matrix valued analytic functions and multivariate prediction can be found in \cite{ArD08} and \cite{ArD18}, respectively.\\
The primary consideration of this paper is de Branges spaces of vector valued entire functions taking values in an infinite dimensional complex separable Hilbert space $\mathfrak{X}$. An initial discussion of de Branges spaces of these vector valued entire functions appeared in the work of L. de Branges and J. Rovnyak \cite{Rovnyak}. The main motivation was to formulate quantum scattering theory in terms of self-adjoint operators and analyze the structure of invariant subspaces in a scattering problem. Hilbert spaces introduced in \cite{Rovnyak} (Theorem $1$) can not be considered the vector generalization of de Branges spaces as they do not generalize Paley-Wiener spaces of vector valued entire functions. 
In a recent work \cite{Mahapatra}, we developed a vector valued generalization of de Branges spaces of entire functions, constructed using a pair of Fredholm operator valued functions. This framework extends both the classical Paley-Wiener spaces of vector valued entire functions and the approach presented in \cite{Rovnyak}. The article also establishes a connection between these generalized spaces and a class of entire operators with infinite deficiency indices introduced by M. G. Krein. Additionally, the study \cite{Mahapatra1} explores analytic Kramer sampling and quasi Lagrange-type interpolation in relation to these spaces, which may be of particular interest. Furthermore, another class of vector valued de Branges spaces constructed from J-contractive operator valued analytic functions and their application to a functional model problem originating from M. G. Krein has been recently investigated in \cite{Bharti}.\\ 
The main objective of this paper is to revive some problems from the theory of de Branges spaces of scalar valued entire functions to the vector valued de Branges spaces setting, which were not considered in \cite{Rovnyak}. We also briefly discuss the difference between de Branges spaces considered in \cite{Rovnyak} and \cite{Mahapatra}. In Theorem \ref{dB operator equality}, we investigate when two de Branges operators generate the same de Branges space. Conversely, in Theorem \ref{dB space equality}, we describe the relationship between their corresponding de Branges operators when the two de Branges spaces coincide. In Theorem $6.3$, we establish a sufficient condition for a $B(\mathfrak{X})$-valued entire function $S(z)$ to be an associated function of the de Branges space $\mathcal{B}(\mathfrak{E})$. Specifically, we show that if there exists $\alpha\in\mathbb{C}$ such that $S(\alpha)$ is invertible, and 
$$\frac{E_+^{-1}S}{\rho_i}u\in H^2_{\mathfrak{X}}(\mathbb{C}_+)~~\mbox{and}~~\frac{E_-^{-1}S}{\rho_{-i}}u\in H^2_{\mathfrak{X}}(\mathbb{C}_+)^\perp \hspace{.3cm}\mbox{for all}~u\in\mathfrak{X},$$
then $S$ is an associated function of $\mathcal{B}(\mathfrak{E})$. Moreover, in Theorem $6.4$, we show that the sufficient condition in Theorem $6.3$ is also necessary, provided there exists $\alpha\in\mathbb{C}$ such that $K_\alpha(\alpha)$ is invertible.\\
Here, we briefly outline the contents of this paper. We recall some preliminary results in section \ref{Sec-2}, which will be helpful in the subsequent sections. In section \ref{Sec-3}, we consider vector valued de Branges spaces and elaborate on several results of de Branges operators. Also, the vector version of problem $45$ from \cite{Branges 4} is discussed. A global factorization of Fredholm operator valued entire functions, invertible at least at one point, are studied in section \ref{Sec-4}, which connects the two de Branges spaces considered in \cite{Rovnyak} and \cite{Mahapatra}. This factorization also provides a conclusion regarding the analytic equivalence of reproducing kernels of de Branges spaces. Section \ref{Sec-5} discusses problem $44$ from \cite{Branges 4} in the vector valued de Branges spaces setting, using the global factorization discussed in the previous section. In section \ref{Sec-6}, we study operator valued entire functions associated with vector valued de Branges spaces and discuss their connection with the operator of multiplication by the independent variable. Finally, section \ref{Sec-7} connects de Branges spaces with the operator nodes.\\
The following notations will be used throughout the paper:\\
$\mathbb{C}_+$ (resp. $\mathbb{C}_-$) denotes the open complex upper (resp. lower) half-plane. $\rho_w(z)=-2\pi i(z-\overline{w})$ for $z,w\in\mathbb{C}$. $\mathfrak{X}$ is a complex separable Hilbert space, and $B(\mathfrak{X})$ is the collection of all bounded linear operators on $\mathfrak{X}$. $I$ denote the identity operator on $\mathfrak{X}$. If $F\in B(\mathfrak{X})$, then $\ker F$, $\mathrm{rng}~F$, and $F^*$ denote the kernel, range, and adjoint of $F$ respectively and
$$Z_F=\{z\in\mathbb{C}:(I-zF)^{-1}\in B(\mathfrak{X})\}.$$
$\Phi(\mathfrak{X})$ denotes the collection of all Fredholm operators on $\mathfrak{X}$, i.e., it is the collection of all closed range bounded linear operators $F$ on $\mathfrak{X}$ such that $\dim(\ker F)<\infty$ and $\dim(\ker F^*)<\infty$. The index of a Fredholm operator $F$ is denoted as $\mathrm{ind} (F)$, and it is an integer obtained by the following relation 
\begin{equation}
 \mathrm{ind} (F)=\dim(\ker F)-\dim(\ker F^*).\label{Index Eq}
 \end{equation}
Since $\ker F^*=(\mathrm{rng}~F)^\perp$, the above relation can be written as
\begin{equation}
\mathrm{ind} (F)=\dim(\ker F)-\dim((\mathrm{rng}~F)^\perp).\label{Index relation}
\end{equation}
Details about Fredholm operators and several applications can be found in \cite{Gohberg1} and \cite{Kato}. Let $\mathcal{H}$ be a reproducing kernel Hilbert space of $\mathfrak{X}$-valued entire functions. For any $\alpha\in\mathbb{C}$, we denote $\mathcal{H}_\alpha=\{f\in\mathcal{H}:f(\alpha)=0\}$. $\mathcal{H}_\alpha$ is a closed subspace of $\mathcal{H}$. The multiplication operator in $\mathcal{H}$ is denoted as $\mathfrak{T}$ and is defined as 
$$\mathfrak{T}f(z)=zf(z)\hspace{.3cm}\mbox{for all}~z\in\mathbb{C}.$$
For any $\alpha\in\mathbb{C}$ and $f$ in a suitable space of $\mathfrak{X}$-valued entire functions, the generalized backward shift operator is denoted as $R_\alpha$ and is defined as 
\begin{equation}
    (R_\alpha f)(z) := \left\{
    \begin{array}{ll}
         \frac{f(z)-f(\alpha)}{z-\alpha}  & \mbox{if } z \neq \alpha \\
         f'(\alpha) & \mbox{if } z = \alpha.
    \end{array} \right.
\end{equation}
We define the operator matrix $\mathcal{I}$ on $\mathfrak{X}\oplus\mathfrak{X}$ as
$$\mathcal{I}=
\begin{bmatrix}
    I  & 0 \\
    0  & -I \\
\end{bmatrix}.
$$
\section{Preliminaries}
\label{Sec-2}
This section briefly recalls some basic spaces of vector valued holomorphic functions, which will be useful in the subsequent sections. A Hilbert space $\mathcal{H}$ of $\mathfrak{X}$-valued entire functions is said to be a reproducing kernel Hilbert space (RKHS) if there exists a function $K_w(z):\mathbb{C}\times\mathbb{C}\to B(\mathfrak{X})$ such that for all $f\in\mathcal{H}$, $w\in\mathbb{C}$ and $u\in\mathfrak{X}$ the following hold:
$$K_wu\in\mathcal{H}\hspace{.3cm}\mbox{and}\hspace{.3cm}\langle f,K_wu\rangle_\mathcal{H}=\langle f(w),u\rangle_\mathfrak{X}.$$
The unique operator valued function $K_w(z)$ is called the reproducing kernel of $\mathcal{H}$. If $\delta_z$ denotes the point evaluation linear operator on $\mathcal{H}$ at the point $z$, i.e., for $f\in\mathcal{H}$, $\delta_z(f)=f(z)$, then $K_w(z)$ can be written as $K_w(z)=\delta_z\delta_w^*$, where $\delta_w^*$ is the adjoint of $\delta_w$. Also, the reproducing kernel of $\mathcal{H}$ is positive. The interpretation of positivity of the reproducing kernel is as follows: for any $p\in\mathbb{N}$, $z_1, z_2,\ldots, z_p\in\mathbb{C}$ and $u_1, u_2, \ldots, u_p\in\mathfrak{X}$,
 $$\sum_{i,j=1}^p \bigg\langle K_{z_j}(z_i)u_j,u_i \bigg\rangle_\mathfrak{X}=\left\lVert\sum_{i=1}^p\delta_{z_i}^*(u_i)\right\rVert_\mathcal{H}^2\geq 0.$$
For a comprehensive study of reproducing kernel Hilbert spaces, see \cite{Paulsen}. A $B(\mathfrak{X})$-valued entire function $S(z)$ is said to be associated to the reproducing kernel Hilbert space $\mathcal{H}$ if $S(\alpha)$ is invertible for some $\alpha\in\mathbb{C}$ and
 \begin{equation}
\frac{f(z)-S(z)S(\alpha)^{-1}f(\alpha)}{z-\alpha}\in\mathcal{H}\hspace{.5cm}\mbox{for all}~f\in\mathcal{H}.
\end{equation}
\begin{lemma}
Let $\mathcal{H}$ be a RKHS of $\mathfrak{X}$-valued entire functions with reproducing kernel $K$. Suppose for some $\alpha\in\mathbb{C}$, $K_\alpha(\alpha)$ is invertible. Then 
\begin{enumerate}
\item
$\mathrm{rng}~K_\alpha(\alpha)=\{f(\alpha):f\in\mathcal{H}\}=\mathfrak{X}.$
\item $S(z)=K_\alpha(z)$ is associated to $\mathcal{H}$ if and only if $R_\alpha\mathcal{H}_\alpha\subseteq \mathcal{H}$, where $\mathcal{H}_\alpha=\{f\in\mathcal{H}:f(\alpha)=0\}$.
\end{enumerate}
\end{lemma}
\begin{proof}
Since $K_\alpha(\alpha)$ is invertible, $\{f(\alpha):f\in\mathcal{H}\}\subseteq \mathrm{rng}~K_\alpha(\alpha)=\mathfrak{X}$. Also, for every $u\in\mathfrak{X}$ as $K_\alpha u\in\mathcal{H}$, we have
$$\mathrm{rng}~K_\alpha(\alpha)=\{K_\alpha(\alpha)u:u\in\mathfrak{X}\}\subseteq \{f(\alpha):f\in\mathcal{H}\}.$$
This proves $(1)$. Now, to prove $(2)$, observe the fact that $f(\cdot)-S(\cdot)S(\alpha)^{-1}f(\alpha)\in\mathcal{H}_\alpha$ for all $f\in\mathcal{H}$.
\end{proof}
$L^2_\mathfrak{X}(\mathbb{R})$ denotes the collection of all $\mathfrak{X}$-valued, weakly measurable, square integrable functions on $\mathbb{R}$. $L^2_\mathfrak{X}(\mathbb{R})$ is a Hilbert space with respect to the following inner product:
$$\langle f,g\rangle=\int_{-\infty}^\infty \langle f(t),g(t)\rangle_\mathfrak{X}~dt\hspace{.5cm}\mbox{for all}~f,g\in L^2_\mathfrak{X}(\mathbb{R}).$$
The Hardy Hilbert space of $\mathfrak{X}$-valued analytic functions on the upper half-plane (resp. lower half-plane) is denoted as $H^2_\mathfrak{X}(\mathbb{C}_+)$ (resp. $H^2_\mathfrak{X}(\mathbb{C}_-)$).
Both $H^2_\mathfrak{X}(\mathbb{C}_+)$ and $H^2_\mathfrak{X}(\mathbb{C}_-)$ are reproducing kernel Hilbert spaces with the corresponding reproducing kernels
$$K_\xi(z)=\frac{I}{\rho_\xi(z)}\hspace{.2cm}\xi,z\in\mathbb{C}_+\hspace{.5cm}\mbox{and}\hspace{.5cm}K_\alpha^{(-)}(\lambda)=-\frac{I}{\rho_\alpha(\lambda)}\hspace{.2cm}\alpha,\lambda\in\mathbb{C}_-.$$

\begin{lemma}
\label{P45}
Suppose $f\in H^2_\mathfrak{X}(\mathbb{C}_+)$ and $\alpha\in\mathbb{C}$ are such that $f$ is holomorphic at $\alpha$, then 
$$\frac{f}{\rho_i}~\mbox{and}~ R_\alpha f~\mbox{belong to}~ H^2_\mathfrak{X}(\mathbb{C}_+).$$
Similarly, suppose $g\in H^2_\mathfrak{X}(\mathbb{C}_-)$ and $\beta\in\mathbb{C}$ are such that $g$ is holomorphic at $\beta$,  then 
$$\frac{g}{\rho_{-i}}~\mbox{and}~ R_\beta g~\mbox{belong to}~ H^2_\mathfrak{X}(\mathbb{C}_-).$$
\end{lemma}
\begin{proof}
The proof follows from Lemma $3.14$ in \cite{ArD18}.
\end{proof}
 The collection of $B(\mathfrak{X})$-valued bounded analytic functions on the upper half-plane is denoted as $H^\infty_{B(\mathfrak{X})}(\mathbb{C}_+)$. This is a Banach space with respect to the following norm:
 $$||F||_\infty=\sup_{y>0}||F(x+iy)||_{B(\mathfrak{X})}<\infty,\hspace{.3cm}F\in H^\infty_{B(\mathfrak{X})}(\mathbb{C}_+).$$
 We denote the closed unit ball of $H^\infty_{B(\mathfrak{X})}(\mathbb{C}_+)$ as $\mathcal{S}$, the operator valued Schur class functions. For a comprehensive study of these spaces, including the existence of their unique boundary functions, see \cite{Rosenblum}. $\mathcal{S}^{in}$ (resp. $\mathcal{S}^{in}_*$) represents the collection of all elements of $\mathcal{S}$ such that their boundary functions are isometries (resp. co-isometries) a.e. on $\mathbb{R}$.\\
 Now, we recall some results about Fredholm operators and Fredholm operator valued analytic functions, which will be crucial for the forthcoming sections. 
 \begin{prop}
 \label{FP}
 Let $M,N\in B(\mathfrak{X})$ be such that $MN\in\Phi(\mathfrak{X})$. Then $M\in\Phi(\mathfrak{X})$ if and only if $N\in\Phi(\mathfrak{X})$.
 \end{prop}
 The proof of the above proposition and several results about Fredholm operators can be found in \cite{Lay} (Theorem $13.1$).
 \begin{thm}{(Fredholm analytic theorem)}
 \label{FAT}
Suppose $\Omega\subseteq \mathbb{C}$ is an open connected set and $F:\Omega\to \Phi(\mathfrak{X})$ is an analytic function. Then either $F(z)$ is not boundedly invertible for any $z\in\Omega$, or there exists a discrete subset $D$ of $\Omega$ such that $F(z)^{-1}\in\Phi(\mathfrak{X})$ for all $z\in\Omega\setminus D$, $F(\cdot)^{-1}$ is analytic on $\Omega\setminus D$ and meromorphic on $\Omega$.
 \end{thm}
More details about this theorem and related topics can be found in \cite{Holden} (Theorem $3.3$) and \cite{Gohberg} (Section $4.1$).
\section{Vector valued de Branges spaces}
\label{Sec-3}
In this section, we recall the definition of vector valued de Branges spaces introduced in \cite{Mahapatra} and discuss several results related to de Branges operators. Let $E_+(z)$ and $E_-(z)$ be two $\Phi(\mathfrak{X})$-valued entire functions such that
\begin{enumerate}
\item  $E_+(z)$ and $E_-(z)$ are invertible for at least at one point, and
 \item $E_+^{-1}E_-\in\mathcal{S}^{in}\cap\mathcal{S}_*^{in}$.
 \end{enumerate}
Then the pair of $\Phi(\mathfrak{X})$-valued entire functions $\mathfrak{E}(z)=(E_-(z),E_+(z))$ is said to be a de Branges operator. Now, a de Branges operator $\mathfrak{E}(z)=(E_-(z),E_+(z))$ provides the following three relations between $E_+(z)$ and $E_-(z)$:
\begin{equation}
E_+(z)E_+(z)^*-E_-(z)E_-(z)^*\succeq 0\hspace{.5cm}\mbox{for all}~z\in\mathbb{C}_+,
\end{equation}
\begin{equation}
E_+(x)E_+(x)^*-E_-(x)E_-(x)^*=0\hspace{.5cm}\mbox{for all}~x\in\mathbb{R}
\end{equation}
and
\begin{equation}
\label{Equality}
E_+(z)E_+(\overline{z})^*-E_-(z)E_-(\overline{z})^*=0\hspace{.5cm}\mbox{for all}~z\in\mathbb{C}.
\end{equation}
Note that the set of Fredholm operators $\Phi(\mathfrak{X})$ is open in $B(\mathfrak{X})$ and the index function is constant in each connected component (see corollary $19.1.6$ in \cite{Hörmander}). Now, since $E_+(z)$ and $E_-(z)$ are Fredholm operator valued entire functions, condition $(1)$ above implies that
\begin{equation}
\mathrm{ind}(E_+(z))=\mathrm{ind}(E_-(z))=0\hspace{.5cm}\mbox{for all}~z\in\mathbb{C}.\label{Index}
\end{equation}
We consider the kernel 
\begin{equation}
K_w(z):= \left\{
    \begin{array}{ll}
         \frac{E_+(z)E_+(w)^*-E_-(z)E_-(w)^*}{\rho_w(z)}  & \mbox{if } z \neq \overline{w} \\
         \frac{E_+^{'} (\overline{w})E_+(w)^*- E_-^{'}(\overline{w})E_-(w)^*}{-2\pi i} & \mbox{if } z = \overline{w}
    \end{array} \right.\label{de Branges kernel}
\end{equation}
corresponding to a de Branges operator $\mathfrak{E}(z)=(E_-(z),E_+(z))$ on $\mathbb{C}\times\mathbb{C}$. The kernel $K_w(z)$ is positive on $\mathbb{C}\times\mathbb{C}$, and we denote the corresponding unique reproducing kernel Hilbert space of $\mathfrak{X}$-valued entire functions as $\mathcal{B}(\mathfrak{E})$.
Another representation of $\mathcal{B}(\mathfrak{E})$ in terms of Hardy spaces is as follows:
\begin{equation}
\mathcal{B}(\mathfrak{E})=\{f:\mathbb{C}\to\mathfrak{X}~\mbox{entire function}~|~ E_+^{-1}f\in H^2_\mathfrak{X}(\mathbb{C}_+),~E_-^{-1}f\in H^2_\mathfrak{X}(\mathbb{C}_+)^\perp\}.
\end{equation}
The proof of the above observation is similar to Theorem $3.10$ in \cite{ArD18}. The norm and inner product defined on $\mathcal{B}(\mathfrak{E})$ are respectively as follows:
\begin{equation}
||f||^2:=\int_{-\infty}^\infty ||E_+^{-1}(t)f(t)||_\mathfrak{X}^2~dt<\infty \label{N1}
\end{equation}
and 
\begin{equation}
\langle f,g \rangle :=\int_{-\infty}^\infty\langle E_+^{-1}(t)f(t),E_+^{-1}(t)g(t)\rangle_\mathfrak{X}~dt, \label{In1}
\end{equation}
where $f,g\in \mathcal{B}(\mathfrak{E})$. The following example of de Branges spaces of vector valued entire functions is actually vector valued Paley-Wiener spaces; more examples of vector valued de Branges spaces can be found in \cite{Mahapatra}. Recall that for any $a>0$, the Paley-Wiener space of $\mathfrak{X}$-valued entire functions is defined as the Hilbert space
$$\{\hat{f}:f~ \mbox{is square integrable and vanishes outside the interval}~ [-a,a]\},$$ 
where $\hat{f}$ denotes the Fourier transformation of $f$. The inner product on this space is given by
$$\langle \hat{f},\hat{g}\rangle=\int_{-\infty}^\infty \langle \hat{f}(t),\hat{g}(t)\rangle_\mathfrak{X}dt.$$
Moreover, the associated reproducing kernel for Paley-Wiener space is
\begin{equation}
\label{Paley}
K_w^a(z)=\frac{\sin(z-\overline{w})a}{\pi(z-\overline{w})}I\hspace{.3cm}\mbox{for all}~w,z\in\mathbb{C}.
\end{equation}
\begin{examp}
Suppose $d$ is any positive real number. Then the de Branges operator $\mathfrak{E}(z)=(E_-(z),E_+(z))$, where 
$$E_+(z)=e^{-i z d}I\hspace{.3cm}\mbox{and}\hspace{.3cm}E_-(z)=e^{i z d}I\hspace{.3cm}\mbox{for all}~z\in\mathbb{C},$$
generates the vector valued Paley-Wiener space having the reproducing kernel $K_w^d(z)$ for all $w,~z\in\mathbb{C}$, as in $(\ref{Paley})$.

\end{examp}
\begin{rmk}
Recall that the de Branges spaces discussed in \cite{Rovnyak} use entire operator valued functions $E_+(z)$ and $E_-(z)$ such that $I-E_+(z)$ and $I-E_-(z)$ are compact for all $z\in\mathbb{C}$. If we stick to this consideration, the previous example, i.e., the Paley-Wiener spaces of vector valued entire functions could not be a de Branges space as whenever $1- e^{-i z d}\neq 0$, $I-E_+(z)$ can not be compact and similarly whenever $1- e^{i z d}\neq 0$, $I-E_-(z)$ can not be compact. In our consideration, the components $E_+(z)$, $E_-(z)$ are considered from a broader class of operator valued entire functions, namely the class of all Fredholm operator valued entire functions. This relaxation is allowing a wider range of spaces to qualify as de Branges spaces. 
\end{rmk}
The following theorem, which is a slightly modified version of Theorem $8.2$ from \cite{Mahapatra}, provides a characterization of the space $\mathcal{B}(\mathfrak{E})$. A similar characterization of de Branges spaces of $\mathbb{C}^n$-valued entire functions and vector valued holomorphic functions on a domain $\Omega\subseteq\mathbb{C}$ can be found in Theorem $7.1$ of \cite{JFA} and Theorem $4.1$ of \cite{Bharti}, respectively.
\begin{thm}
\label{Characterization}
Let $\mathcal{H}$ be a RKHS of $\mathfrak{X}$-valued entire functions with $B(\mathfrak{X})$-valued RK $K_w(z)$ and suppose $\beta\in\mathbb{C}_+$ be such that
$$K_\beta(z),K_{\overline{\beta}}(z)\in\Phi(\mathfrak{X})\hspace{.4cm}\mbox{for all}~z\in\mathbb{C}$$
and
$$K_\beta(\beta),~K_{\overline{\beta}}(\overline{\beta})\hspace{.4cm}\mbox{both are invertible}.$$
Then the RKHS $\mathcal{H}$ coincides with a de Branges space $\mathcal{B}(\mathfrak{E})$ isometrically if and only if
\begin{enumerate}
\item $R_\beta\mathcal{H}_\beta\subseteq \mathcal{H}$, $R_{\overline{\beta}}\mathcal{H}_{\overline{\beta}}\subseteq \mathcal{H}$, and
\item $(\mathfrak{T}-\overline{\beta}I)R_\beta:\mathcal{H}_\beta\to\mathcal{H}_{\overline{\beta}}$ is an isometric isomorphism.
\end{enumerate}
Moreover, in this case, the de Branges operator $\mathfrak{E}(z)=(E_-(z),E_+(z))$ takes the following form
\begin{equation}
\label{E_+}
E_+(z)=\rho_\beta(z)\rho_\beta(\beta)^{-\frac{1}{2}}K_\beta(z)K_\beta(\beta)^{-\frac{1}{2}},
\end{equation}
and 
\begin{equation}
\label{E_-}
E_-(z)=-\rho_{\overline{\beta}}(z)\rho_\beta(\beta)^{-\frac{1}{2}}K_{\overline{\beta}}(z)K_{\overline{\beta}}(\overline{\beta})^{-\frac{1}{2}}.
\end{equation}
\end{thm}
\begin{proof}
The proof of this theorem follows similarly to the proof of Theorem $8.2$ in \cite{Mahapatra} except for the necessary part where we need to prove $E_+(\beta)$ and $E_-(\overline{\beta})$ both are surjective.
Indeed, if the RKHS $\mathcal{H}$ satisfies the two sufficient conditions $(1)$ and $(2)$, and if we define $E_+(z)$ and $E_-(z)$ as in $(\ref{E_+})$ and $(\ref{E_-})$, respectively, then $E_+(z)$ and $E_-(z)$ are $\Phi(\mathfrak{X})$-valued entire functions, and $E_+(\beta)=\rho_\beta(\beta)^{\frac{1}{2}}K_\beta(\beta)^{\frac{1}{2}}$ and $E_-(\overline{\beta})=\rho_\beta(\beta)^{\frac{1}{2}}K_{\overline{\beta}}(\overline{\beta})^{\frac{1}{2}}$ both of which are invertible. For $z\neq \overline{w}$, we compute
$$K_w(z)=\frac{E_+(z)E_+(w)^*-E_-(z)E_-(w)^*}{\rho_w(z)},$$
which implies 
$$E_+(z)E_+(z)^*-E_-(z)E_-(z)^*=\rho_z(z)K_z(z)\succeq 0\hspace{.3cm}\mbox{for all}~z\in\mathbb{C}_+,$$
and
$$E_+(z)E_+(z)^*-E_-(z)E_-(z)^*=0\hspace{.3cm}\mbox{for all}~z\in\mathbb{R}.$$
Thus, $E_+^{-1}E_-\in\mathcal{S}^{in}\cap\mathcal{S}_*^{in}$, and consequently, $\mathfrak{E}(z)=(E_-(z),E_+(z))$ is a de Branges operator. Since the reproducing kernels of $\mathcal{H}$ and $\mathcal{B}(\mathfrak{E})$ coincide, we conclude that $\mathcal{H}$ is the de Branges space $\mathcal{B}(\mathfrak{E})$.\\
Conversely, assume that the RKHS $\mathcal{H}$ is a de Branges space $\mathcal{B}(\mathfrak{E})$. Since both $K_\beta(\beta)$ and $K_{\overline{\beta}}(\overline{\beta})$ are invertible, it follows that $E_+(\beta)^*$ and $E_-(\overline{\beta})^*$ are injective. Now $E_+(\beta)$ and $E_-(\overline{\beta})$ are also injective follows from  $(\ref{Index Eq})$. This implies that $E_+(\beta)$ and $E_-(\overline{\beta})$ are surjective. Notice that here, we do not insist that $E_+(\beta)$ and $E_-(\overline{\beta})$ are both self-adjoint operators; instead, we are using the fact that they are both Fredholm operators with index zero. The rest of the proof follows from the necessary part of Theorem $8.2$ in \cite{Mahapatra}.
\end{proof}
The following two theorems discuss the fact that corresponding to different de Branges operators, we can get the same de Branges space. This situation is characterized by the $\mathcal{I}$-unitary operator matrix on $\mathfrak{X}\oplus\mathfrak{X}$. A similar discussion for the matrix case can be found in \cite{ArD18} (Chapter $3.2$). Suppose 
\begin{equation}
\label{Operator matrix}
U=
\begin{bmatrix}
    U_{11}  & U_{12} \\
    U_{21}  & U_{22} \\
\end{bmatrix}
\end{equation}
be an operator matrix such that all its entries belong to $B(\mathfrak{X})$. 
\begin{defn}
Let $\mathfrak{H}$ be a Hilbert space and $J\in B(\mathfrak{H})$ be a signature operator, i.e., $$J=J^*=J^{-1}.$$ An operator $A\in B(\mathfrak{H})$ is called $J$-unitary if it satisfies 
\begin{equation}
\label{uni}
AJA^*=A^*JA=J.
\end{equation}
\end{defn}
Since $\mathcal{I}$ is a signature operator on $\mathfrak{X}\oplus\mathfrak{X}$, $U$ is a $\mathcal{I}$-unitary operator matrix if it satisfies $(\ref{uni})$. The following lemma gives some insightful information about the entries of an $\mathcal{I}$-unitary operator matrix. The proof of this lemma follows from \cite{Alpay1} (Lemma $5.2$). 
\begin{lemma}
\label{Unitary lemma}
If $U$ is an $\mathcal{I}$-unitary operator matrix as in $(\ref{Operator matrix})$, the following assertions are true:
\begin{enumerate}
\item $U_{11}$ and $U_{22}$ are invertible operators.
\item $||U_{12}U_{22}^{-1}||<1$, $||U_{21}U_{11}^{-1}||<1$, $||U_{11}^{-1}U_{12}||<1$ and $||U_{22}^{-1}U_{21}||<1$.
\end{enumerate}
\end{lemma} 
Now, we mention an immediate corollary of the previous lemma and Theorem $3$ in \cite{Harris}, which gives a factorization of the $\mathcal{I}$-unitary operator matrix $U$.
\begin{cor}
If $U$ is an $\mathcal{I}$-unitary operator matrix as in $(\ref{Operator matrix})$, and
\[A=
\begin{bmatrix}
    (I-K^*K)^{-\frac{1}{2}}  & K^*(I-KK^*)^{-\frac{1}{2}} \\
    (I-KK^*)^{-\frac{1}{2}}K  & (I-KK^*)^{-\frac{1}{2}} \\
\end{bmatrix}
\]
 then the following factorizations hold:
\begin{enumerate}
\item $U=LA$, where $K=U_{22}^{-1}U_{21}$ and
$$L=
\begin{bmatrix}
    (U_{11}-U_{12}K)(I-K^*K)^{-\frac{1}{2}}  & 0 \\
    0  & U_{22}(I-KK^*)^{\frac{1}{2}} \\
\end{bmatrix}.
$$
\item $U=AR$, where $K=U_{21}U_{11}^{-1}$ and
$$R=
\begin{bmatrix}
    (I-K^*K)^{\frac{1}{2}}U_{11}  & 0 \\
    0  & (I-KK^*)^{-\frac{1}{2}}(U_{22}-KU_{12}) \\
\end{bmatrix}.
$$
\end{enumerate}
\end{cor}
In the following two theorems, we discuss conditions under which the de Branges spaces $\mathcal{B}(\mathfrak{E})$ and $\mathcal{B}(\mathfrak{F})$, corresponding to two different de Branges operators $\mathfrak{E}$ and $\mathfrak{F}$, respectively, coincide isometrically.
\begin{thm}
\label{dB operator equality}
Let $\mathcal{B}(\mathfrak{E})$ be a de Branges space of $\mathfrak{X}$-valued entire functions corresponding to a de Branges operator $\mathfrak{E}(z)=(E_-(z),E_+(z))$ and $U$ be the $\mathcal{I}$-unitary operator matrix as in $(\ref{Operator matrix})$ such that $$F_-(z)=E_-(z)~U_{11}+E_+(z)~U_{21} ,\hspace{.3cm} F_+(z)=E_-(z)~U_{12}+E_+(z)~U_{22}$$ belong to $\Phi(\mathfrak{X})$ for all $z\in\mathbb{C}$. Then $\mathfrak{F}(z)=(F_-(z),F_+(z))$ is a de Branges operator and $\mathcal{B}(\mathfrak{E})=\mathcal{B}(\mathfrak{F})$.
\end{thm}
\begin{proof}
$F_+(z)$ and $F_-(z)$ both are $\Phi(\mathfrak{X})$-valued entire functions. Fix any $\beta\in\mathbb{C}_+$ to be such that $E_+(\beta)$ and $E_-(\overline{\beta})$ both are invertible. Then $||E_+^{-1}(\beta)E_-(\beta)||\leq 1$ together with Lemma \ref{Unitary lemma} implies that
$$F_+(\beta)=E_+(\beta)[E_+^{-1}(\beta)E_-(\beta)U_{12}U_{22}^{-1}+I]U_{22}$$
is invertible. Similarly, it can be proved that $F_-({\overline{\beta}})$ is invertible. Also, it only needs a small calculation to show that, for every $z,w\in\mathbb{C}$,
$$F_+(z)F_+(w)^*-F_-(z)F_-(w)^*=E_+(z)E_+(w)^*-E_-(z)E_-(w)^*.$$ 
This implies $F_+^{-1}F_-\in\mathcal{S}^{in}\cap\mathcal{S}_*^{in}$ and $\mathcal{B}(\mathfrak{E})=\mathcal{B}(\mathfrak{F})$. 
\end{proof}
The following theorem gives the converse of this result.
\begin{thm}
\label{dB space equality}
Let $\mathfrak{E}(z)=(E_-(z),E_+(z))$ and $\mathfrak{F}(z)=(F_-(z),F_+(z))$ be two de Branges operators, and the corresponding de Branges spaces are $\mathcal{B}(\mathfrak{E})$ and $\mathcal{B}(\mathfrak{F})$ with reproducing kernels $K_w^\mathfrak{E}(z)$ and $K_w^\mathfrak{F}(z)$ respectively. Also, suppose that there exists $\beta\in\mathbb{C}_+$ be such that $K_\beta^\mathfrak{E}(\beta)$ and $K_{\overline{\beta}}^\mathfrak{E}(\overline{\beta})$ both are invertible. Then $\mathcal{B}(\mathfrak{E})=\mathcal{B}(\mathfrak{F})$ implies that there exists an $\mathcal{I}$-unitary operator matrix $U$ on $\mathfrak{X}\oplus\mathfrak{X}$ such that 
$$[F_-(z)~F_+(z)]=[E_-(z)~E_+(z)]~U\hspace{.3cm}\mbox{for all}~z\in\mathbb{C}.$$
\end{thm}
\begin{proof}
$\mathcal{B}(\mathfrak{E})=\mathcal{B}(\mathfrak{F})$ implies that $K_w^\mathfrak{E}(z)=K_w^\mathfrak{F}(z)$ for all $z,w\in\mathbb{C}$. Also, $K_\beta^\mathfrak{E}(\beta)$ and $K_{\overline{\beta}}^\mathfrak{E}(\overline{\beta})$ both are invertible gives the following:
\begin{enumerate}
\item $E_+(\beta)$ and $E_-(\overline{\beta})$ are invertible operators, and
\item $||E_+^{-1}(\beta)E_-(\beta)||<1$ and $||E_-^{-1}(\overline{\beta})E_+(\overline{\beta})||<1$.
\end{enumerate}
Consider the following two operator matrices on $\mathfrak{X}\oplus\mathfrak{X}$,
$$U_\mathfrak{E}(\beta)=
\begin{bmatrix}
    E_-(\beta)  & E_+(\beta) \\
    E_-(\overline{\beta})  & E_+(\overline{\beta}) \\
\end{bmatrix}
\hspace{.3cm}\mbox{and}\hspace{.3cm} U_\mathfrak{F}(\beta)=
\begin{bmatrix}
     F_-(\beta) & F_+(\beta) \\
    F_-(\overline{\beta})  & F_+(\overline{\beta}) \\
\end{bmatrix}.
$$
It is clear that these operator matrices follow the identity
$$U_\mathfrak{E}(\beta)\mathcal{I}U_\mathfrak{E}(\beta)^*=
\begin{bmatrix}
    -\rho_\beta(\beta)K_\beta^\mathfrak{E}(\beta)  & 0 \\
    0  & -\rho_{\overline{\beta}}(\overline{\beta})K_{\overline{\beta}}^\mathfrak{F}(\overline{\beta}) \\
\end{bmatrix}
=U_\mathfrak{F}(\beta)\mathcal{I}U_\mathfrak{F}(\beta)^*.$$
From the Schur complement formula, it can be seen that the operator matrix $U_\mathfrak{E}(\beta)$ is invertible if and only if $E_-(\overline{\beta})-E_+(\overline{\beta})E_+^{-1}(\beta)E_-(\beta)$ is an invertible operator. Now
$$E_-(\overline{\beta})-E_+(\overline{\beta})E_+^{-1}(\beta)E_-(\beta)=E_-(\overline{\beta})[I-E_-^{-1}(\overline{\beta})E_+(\overline{\beta})E_+^{-1}(\beta)E_-(\beta)]$$
implies that $U_\mathfrak{E}(\beta)$ is invertible. Similarly, it can be proved that $U_\mathfrak{F}(\beta)$ is also invertible. Consider the operator matrix $U_0=U_\mathfrak{E}(\beta)^*(U_\mathfrak{F}(\beta)^*)^{-1}$. It is clear that $U_0$ is $\mathcal{I}$-unitary and 
$$[F_-(z)~F_+(z)]\mathcal{I}U_\mathfrak{F}(\beta)^*=[E_-(z)~E_+(z)]\mathcal{I}U_\mathfrak{E}(\beta)^*\hspace{.3cm}\mbox{for all}~z\in\mathbb{C}.$$
Now, choose $U=\mathcal{I}U_0\mathcal{I}$, then $U$ is an $\mathcal{I}$-unitary operator and 
$$[F_-(z)~F_+(z)]=[E_-(z)~E_+(z)]~U\hspace{.3cm}\mbox{for all}~z\in\mathbb{C}.$$
\end{proof}
The following lemma provides a connection between the real poles of the meromorphic functions $E_+^{-1}(\cdot)$ and $E_-^{-1}(\cdot)$ when $\mathfrak{E}(z)=(E_-(z),E_+(z))$ is a de Branges operator.
\begin{lemma}
\label{Pole}
Let $\mathcal{B}(\mathfrak{E})$ be a de Branges space of $\mathfrak{X}$-valued entire functions corresponding to a de Branges operator $\mathfrak{E}(z)=(E_-(z),E_+(z))$. Then, a point $x\in\mathbb{R}$ is a pole of $E_+^{-1}(\cdot)$ if and only if it is a pole of $E_-^{-1}(\cdot)$.
\end{lemma}
\begin{proof}
In view of Theorem \ref{FAT}, to prove this lemma, it is sufficient to show that for any $x\in\mathbb{R}$, $E_+(x)$ is invertible if and only if $E_-(x)$ is invertible. Recall that for all $x\in\mathbb{R}$, $E_+$ and $E_-$ satisfy the following identity:
$$E_+(x)E_+(x)^*=E_-(x)E_-(x)^*.$$
Now, a theorem due to Douglas (see \cite{Douglas}) gives $\mathrm{rng}E_+(x)=\mathrm{rng}E_-(x)$ for all $x\in\mathbb{R}$. Thus $E_+(x)$ is invertible implies $\mathrm{rng}E_+(x)=\mathfrak{X}$ and $\ker E_+(x)=\{0\}$. Therefore, $\mathrm{rng}E_-(x)=\mathfrak{X}$ and the equality implies 
$$\ker E_-(x)^*=\ker E_+(x)^*=\{0\}.$$
Now, since $\mathrm{ind}(E_-(x))=0$, we have $\ker E_-(x)=\{0\}$, which gives $E_-(x)$ is invertible. Similarly, when $E_-(x)$ is invertible, $E_+(x)$ is invertible.
\end{proof}
The following lemma can be considered the vector generalization of Problem $45$ from the book \cite{Branges 4}.
\begin{lemma}
Let $\mathcal{B}(\mathfrak{E})$ be a de Branges space of $\mathfrak{X}$-valued entire functions corresponding to the de Branges operator $\mathfrak{E}(z)=(E_-(z),E_+(z))$. Assume that $f\in\mathcal{B}(\mathfrak{E})$ and $\alpha$ is any complex number such that $f(\alpha)=0$. Then the following implications hold:
\begin{enumerate}
\item $\frac{f(z)}{z-\alpha}\in\mathcal{B}(\mathfrak{E})$ for all $\alpha\in\mathbb{C}\setminus\mathbb{R}$, and
\item $\frac{f(z)}{z-\alpha}\in\mathcal{B}(\mathfrak{E})$ for all $\alpha\in\mathbb{R}$, where $E_+(\alpha)$ (so too $E_-(\alpha)$) is invertible.
\end{enumerate} 
\end{lemma}
\begin{proof}
We divide the proof into the following three cases:\\
$\underline{\mbox{Case}~1~(\mbox{When}~\alpha\in\mathbb{C}_+):}$ \\
Since $f\in\mathcal{B}(\mathfrak{E})$, $E_+^{-1}f\in H_\mathfrak{X}^2(\mathbb{C}_+)$ and $E_-^{-1}f\in H_\mathfrak{X}^2(\mathbb{C}_+)^\perp$. To show that $\frac{f(z)}{z-\alpha}\in\mathcal{B}(\mathfrak{E})$ we need to verify that $E_+^{-1}(z)\frac{f(z)}{z-\alpha}\in H_\mathfrak{X}^2(\mathbb{C}_+)$ and $E_-^{-1}(z)\frac{f(z)}{z-\alpha}\in H_\mathfrak{X}^2(\mathbb{C}_+)^\perp$. Now, 
$$E_-^{-1}(z)\frac{f(z)}{z-\alpha}=\frac{z-i}{z-\alpha}\frac{E_-^{-1}(z)f(z)}{z-i}.$$
Since $\frac{z-i}{z-\alpha}$ is bounded and analytic on $\mathbb{C}_-$ and $\frac{E_-^{-1}(z)f(z)}{z-i}\in H_\mathfrak{X}^2(\mathbb{C}_+)^\perp$, it is clear that $E_-^{-1}(z)\frac{f(z)}{z-\alpha}\in H_\mathfrak{X}^2(\mathbb{C}_+)^\perp$. Similarly, we write 
$$E_+^{-1}(z)\frac{f(z)}{z-\alpha}=\frac{z+i}{z-\alpha}\frac{E_+^{-1}(z)f(z)}{z+i}.$$
Since $\frac{z+i}{z-\alpha}$ is not analytic and bounded on $\mathbb{C}_+$, we can not argue as before. But using the Cauchy integral formula and the facts that $\frac{1}{z+i}\in H^2(\mathbb{C}_+)$, $E_+^{-1}(z)f(z)\in H^2_\mathfrak{X}(\mathbb{C}_+)$, it can be proved that $E_+^{-1}(z)\frac{f(z)}{z-\alpha}\in H^2_\mathfrak{X}(\mathbb{C}_+)$. For the supporting calculations, see Lemma $3.14$ in \cite{ArD18}.\\
$\underline{\mbox{Case}~2~(\mbox{When}~\alpha\in\mathbb{C}_-):}$ \\
Since this case can be proved similar to the first one, we avoid the calculations.\\
$\underline{\mbox{Case}~3~(\mbox{When}~\alpha\in\mathbb{R}):}$ \\
Observe from Lemma \ref{Pole} that for any $\alpha\in\mathbb{R}$, $E_+(\alpha)$ is invertible if and only if $E_-(\alpha)$ is invertible.  When $E_+(\alpha)$ is invertible, $E_+^{-1}f$ and $E_-^{-1}f$ are analytic at $\alpha$. Now, the remaining argument follows from Lemma \ref{P45}.
\end{proof}
\section{Global Factorization and Analytic Equivalence}
\label{Sec-4}
This section presents a global factorization of Fredholm operator valued entire functions, which are invertible at least at one point. This factorization then provides a connection between the de Branges spaces considered in \cite{Rovnyak} (Theorem $1$) and in \cite{Mahapatra}. Also, a result of analytic equivalence between the corresponding two reproducing kernels of de Branges spaces is concluded. Let $A(z)$ be a $B(\mathfrak{X})$ valued entire function such that $A(z)\in \Phi(\mathfrak{X})$ for all $z\in\mathbb{C}$. Suppose $A(z)$ is invertible at $z_0\in\mathbb{C}$, then because of Theorem \ref{FAT} there exists a discrete set $D=\{z_1, z_2,\ldots,z_n,\ldots\}\subset\mathbb{C}$ such that $A(z)$ are invertible for all $z\in\mathbb{C}\setminus D$. The case when $D$ is finite, a global factorization was discussed in \cite{Bart}, and local factorizations of $A(z)$ can be found in \cite{Holden}. Here, we consider the case when $D$ is infinite and thus $|z_n|\to \infty$ as $n\to \infty$.\\
 The following theorem is due to L. de Branges and J. Rovnyak ( Theorem $19$, Appendix \cite{Rovnyak} ). 
 \begin{thm}
 \label{product}
 Let $\{P_n\}$ be a sequence of finite rank orthogonal projections and $\{z_n\}_1^\infty$ be a sequence of complex numbers such that $|z_n|\to\infty$ as $n\to\infty$. Suppose a complex number $z_0$ is such that $z_n-z_0\neq 0$ for all $n$. Then
 \begin{multline}
 \label{P1}
 P(z)=\lim_{n\to\infty}\left[I-\frac{z-z_0}{z_1-z_0}P_1\right]\exp\left[\frac{z-z_0}{z_1-z_0}P_1\right]\ldots\\
 \left[I-\frac{z-z_0}{z_n-z_0}P_n\right]\exp\left[\frac{z-z_0}{z_n-z_0}P_n+\ldots+\frac{1}{n}\frac{(z-z_0)^n}{(z_n-z_0)^n}P_n\right]
 \end{multline}
 converges uniformly in any bounded set with respect to the operator norm and $P(z)$ is a $B(\mathfrak{X})$ valued entire function such that $I-P(z)$ is compact operator for all $z\in\mathbb{C}$. Moreover, for all $z\in\mathbb{C}\setminus\{z_n\}_1^\infty$, $P(z)$ is invertible. 
 \end{thm}
 The next theorem provides a global factorization of Fredholm operator valued entire functions that are invertible at least at one point. This theorem generalizes Theorem $20$ from the appendix in \cite{Rovnyak}. Since the proof follows a similar line of reasoning, we have included it as an appendix at the end of this article.
 \begin{thm}
 \label{AT}
 Let $A(z)$ be a $B(\mathfrak{X})$ valued entire function such that $A(z)\in\Phi(\mathfrak{X})$ for all $z\in\mathbb{C}$ and $A(z_0)$ is invertible for some $z_0\in\mathbb{C}$. Then $A$ can be factored as
 \begin{equation}
 A(z)=P(z)~E(z)=F(z)~Q(z),\label{AE}
 \end{equation}
 where $P(z)$ and $Q(z)$ are operator valued entire functions of the form $(\ref{P1})$ and $E(z), F(z)$ are invertible operator valued entire functions.
 \end{thm}
 
The above theorem can be considered as the operator analog of the well known Weierstrass factorization theorem for scalar valued entire functions. Also, for any matrix valued entire function $A(z)$, a factorization of the form $(\ref{AE})$ can be readily derived; here, a point $z_0\in\mathbb{C}$ is considered a zero of $A(z)$ if the determinant of $A(z_0)$ is equal to zero. Additionally, in this direction, it is worth mentioning some other factorizations of matrix valued analytic functions available in the literature. For instance, the seminal work of Potapov \cite{Potapov} regarding the factorization of matrix valued inner functions and its application to multiplicative representations of matrix valued analytic functions. Recently, this was extended in \cite{Curto} for operator valued inner functions. The factorizations of Potapov for J-contractive and J-inner matrix valued functions and their applications are also available in \cite{ArD08}. A Hadamard factorization for matrix valued entire functions can be found in \cite{Ran}. The following proposition gives a sense of how the de Branges spaces considered in \cite{Rovnyak} are connected with the de Branges spaces under consideration in this paper.
\begin{prop}
\label{Connection}
Suppose $\mathfrak{E}(z)=(E_-(z),E_+(z))$ is a de Branges operator having the following factorizations of both $E_-(z)$ and $E_+(z)$ as deduced in the previous theorem:
\begin{enumerate}
\item $E_-(z)=F_-(z)Q_-(z)$ for all $z\in\mathbb{C}$, where $I-Q_-(z)$ is compact and $F_-(z)$ is invertible for all $z\in\mathbb{C}$. Also, $Q_-(z)$ is invertible, whereas $E_-(z)$ is invertible.
\item $E_+(z)=F_+(z)Q_+(z)$ for all $z\in\mathbb{C}$, where $I-Q_+(z)$ is compact and $F_+(z)$ is invertible for all $z\in\mathbb{C}$. Also, $Q_+(z)$ is invertible, whereas $E_+(z)$ is invertible.
\end{enumerate}
Then, if $F_-(z)=F_+(z)$ for all $z\in\mathbb{C}$, the pair of Fredholm operator valued entire functions $\mathfrak{Q}(z)=(Q_-(z),Q_+(z))$ is a de Branges operator, and $\mathcal{B}(\mathfrak{E})=F_+\mathcal{B}(\mathfrak{Q})$.
\end{prop}
\begin{proof}
Observe that
\begin{align}
K_\xi(z)&=\frac{E_+(z)E_+(\xi)^*-E_-(z)E_-(\xi)^*}{\rho_\xi(z)}\nonumber\\
& = F_+(z)\left[\frac{Q_+(z)Q_+(\xi)^*-Q_-(z)Q_-(\xi)^*}{\rho_\xi(z)}\right]F_+(\xi)^*\nonumber\\
& = F_+(z)\Gamma_\xi(z)F_+(\xi)^*,\label{D1}
\end{align}
where $\Gamma_\xi(z)=\frac{Q_+(z)Q_+(\xi)^*-Q_-(z)Q_-(\xi)^*}{\rho_\xi(z)}$. From (\ref{D1}), it is clear that $\Gamma_\xi(z)$ is a positive kernel, and the pair of Fredholm operator valued entire functions $\mathfrak{Q}(z)=(Q_-(z),Q_+(z))$ is a de Branges operator. Let $u\in\mathfrak{X}$ and for some $\xi\in\mathbb{C}$, $v=F_+(\xi)^*u$, then the following linear map
$$\Gamma_\xi v\mapsto K_\xi u=F_+(\cdot)\Gamma_\xi v$$
between $\mathcal{B}(\mathfrak{Q})$ and $\mathcal{B}(\mathfrak{E})$ proves that $\mathcal{B}(\mathfrak{E})=F_+\mathcal{B}(\mathfrak{Q})$.
\end{proof}
\begin{rmk}
Observe that the above proposition implies that, given any de Branges space $\mathcal{B}(\mathfrak{E})$ as in our present consideration, if $E_+(z)$ and $E_-(z)$ have a common factor $F_+(z)$ that is invertible for all $z\in\mathbb{C}$, then $\mathcal{B}(\mathfrak{E})$ is canonically isomorphic to a de Branges space $\mathcal{B}(\mathfrak{Q})$ as considered in \cite{Rovnyak}, differing only by the fixed invertible factor $F_+(z)$. However, in general, the hypothesis that $F_+(z)=F_-(z)$ for all $z\in\mathbb{C}$ need not be true. In the next section (Theorem \ref{Embedding}), we shall see that though a situation of equal factors occurs, the factor is having real zeros of $E_+(z)$ ( same as $E_-(z)$).
\end{rmk}
 
\begin{rmk}
Observe that $(\ref{D1})$ gives an analytic equivalence between the two reproducing kernels $K_\xi(z)$ and $\Gamma_\xi(z)$, i.e.,
 $$K_z(z)=F_+(z)\Gamma_z(z)F_+(z)^*\hspace{.3cm}\mbox{for all}~z\in\mathbb{C}.$$
 Details about analytic equivalence of two operator valued entire functions can be found in \cite{Gohberg1} and \cite{Gohberg}.
\end{rmk}
\section{Isometric Embedding}
\label{Sec-5}
In this section, we deal with several isometric embedding results related to the vector valued de Branges spaces $\mathcal{B}(\mathfrak{E})$. In particular, we present the vector version of Problem $44$ from the book \cite{Branges 4}. Theorem \ref{Embedding} is the main result of this section. Moreover, assume that the de Branges operators involved in this section satisfy the following additional two conditions:
\begin{Hypo}
\label{H1}
Suppose $\mathfrak{E}(z)=(E_-(z),E_+(z))$  is a de Branges operator, then
\begin{enumerate}
\item  $E_+(z)$ is invertible for all $z\in\mathbb{C}_+$,
 \item $E_-(z)$ is invertible for all $z\in\mathbb{C}_-$.
 \end{enumerate}
\end{Hypo}
The following two lemmas are motivated by \cite{Woracek}, where the de Branges spaces under consideration were Hilbert spaces of scalar valued entire functions.
\begin{lemma}
\label{Isometry}
Let $\mathcal{B}(\mathfrak{E})$ and $\mathcal{B}(\mathfrak{E}^0)$ be two de Branges spaces corresponding to the de Branges operators $\mathfrak{E}(z)=(E_-(z),E_+(z))$ and $\mathfrak{E}^0(z)=(E_-^0(z),E_+^0(z))$ respectively. Suppose $P(z)$ is a $B(\mathfrak{X})$-valued entire function such that 
\begin{equation}
E_+(z)=P(z)E_+^0(z)\hspace{.5cm}\mbox{for all}~z\in\mathbb{C},\label{W1}
\end{equation}
and
\begin{equation}
E_-(z)=P(z)E_-^0(z)\hspace{.5cm}\mbox{for all}~z\in\mathbb{C}.\label{W2}
\end{equation}
Then $P\mathcal{B}(\mathfrak{E}^0)$ is contained in $\mathcal{B}(\mathfrak{E})$ isometrically.
\end{lemma}
\begin{proof}
From (\ref{W1}) and (\ref{W2}) it is clear that $P(z)$ is $\Phi(\mathfrak{X})$-valued entire function and is invertible for all $z\in\mathbb{C}_+\cup\mathbb{C}_-$. Also, due to the Fredholm analytic theorem, we conclude that $P(z)$ is invertible for all $z\in\mathbb{R}$ except possibly on a discrete set.  Now, suppose $f\in \mathcal{B}(\mathfrak{E}^0)$, i.e., $(E_+^0)^{-1}f\in H^2_\mathfrak{X}(\mathbb{C}_+)$ and $(E_-^0)^{-1}f\in H^2_\mathfrak{X}(\mathbb{C}_+)^\perp$. To show the isometric containment of $P\mathcal{B}(\mathfrak{E}^0)$ in $\mathcal{B}(\mathfrak{E})$, it is sufficient to show that $Pf\in \mathcal{B}(\mathfrak{E})$ and $||Pf||_{\mathcal{B}(\mathfrak{E})}=||f||_{\mathcal{B}(\mathfrak{E}^0)}$. The following supplementary calculations prove the lemma.
$$E_+^{-1}(z)P(z)f(z)=(E_+^0)^{-1}(z)P(z)^{-1}P(z)f(z)=(E_+^0)^{-1}(z)f(z)\hspace{.3cm}\forall z\in\mathbb{C}_+,$$
$$E_-^{-1}(z)P(z)f(z)=(E_-^0)^{-1}(z)P(z)^{-1}P(z)f(z)=(E_-^0)^{-1}(z)f(z)\hspace{.3cm}\forall z\in\mathbb{C}_-$$
and
\begin{align*}
||Pf||_{\mathcal{B}(\mathfrak{E})} & =\int_{-\infty}^\infty ||E_+^{-1}(t)P(t)f(t)||^2 dt \\
& = \int_{-\infty}^\infty ||(E_+^o)^{-1}(t)P(t)^{-1}P(t)f(t)||^2 dt \\
& = \int_{-\infty}^\infty ||(E_+^o)^{-1}(t)f(t)||^2 dt \\
& =||f||_{\mathcal{B}(\mathfrak{E}^0)}.
\end{align*}
\end{proof}
The following lemma is an application of the previous lemma.
\begin{lemma}
Let $\mathcal{B}(\mathfrak{E})$ and $\mathcal{B}(\mathfrak{F})$ be two de Branges spaces corresponding to the de Branges operators $\mathfrak{E}(z)=(E_-(z),E_+(z))$ and $\mathfrak{F}(z)=(F_-(z),F_+(z))$ respectively. Assume that for all $z\in\mathbb{C}$, the following four relations are true:
\begin{align}
F_+(z)E_+(z)& =E_+(z)F_+(z);\label{a}\\  
F_+(z)E_-(z)& = E_-(z)F_+(z);\label{b}\\  
F_-(z)E_+(z)& = E_+(z)F_-(z); \label{c}\\ 
F_-(z)E_-(z)& = E_-(z)F_-(z).  \label{d}
\end{align}
Then, the pair of Fredholm operator valued entire functions 
$$\mathfrak{E}\mathfrak{F}(z)=(E_-(z)F_-(z),E_+(z)F_+(z))$$ 
will be a de Branges operator and 
\begin{equation}
\mathcal{B}(\mathfrak{E}\mathfrak{F})=E_+\mathcal{B}(\mathfrak{F})\oplus F_-\mathcal{B}(\mathfrak{E}).
\end{equation}
\end{lemma}
\begin{proof}
Only after a few small calculations, it can be easily shown that $$(E_+F_+)^{-1}(E_-F_-)\in\mathcal{S}^{in}\cap\mathcal{S}_*^{in}.$$
This implies that $\mathfrak{E}\mathfrak{F}(z)=(E_-(z)F_-(z),E_+(z)F_+(z))$ is a de Branges operator, and $\mathcal{B}(\mathfrak{E}\mathfrak{F})$ is the corresponding de Branges space. If $K_w^\mathfrak{E}(z)$ is the reproducing kernel of $\mathcal{B}(\mathfrak{E})$  and $K_w^\mathfrak{F}(z)$ is the reproducing kernel of $\mathcal{B}(\mathfrak{F})$, the reproducing kernel $K_w(z)$ of $\mathcal{B}(\mathfrak{E}\mathfrak{F})$ can be represented as follows:
\begin{align*}
K_w(z)& = \frac{E_+(z)F_+(z)F_+(w)^*E_+(w)^*-E_-(z)F_-(z)F_-(w)^*E_-(w)^*}{\rho_w(z)}\\
& = E_+(z)K_w^\mathfrak{F}(z)E_+(w)^*+F_-(z)K_w^\mathfrak{E}(z)F_-(w)^*.
\end{align*}
Now, we show that $E_+\mathcal{B}(\mathfrak{F})$ is isometrically contained in $\mathcal{B}(\mathfrak{E}\mathfrak{F})$. We follow the same technique as in the previous lemma. Specifically, we show that $E_+f\in \mathcal{B}(\mathfrak{E}\mathfrak{F})$ for all $f\in \mathcal{B}(\mathfrak{F})$ and $||E_+f||_{\mathcal{B}(\mathfrak{E}\mathfrak{F})}=||f||_{\mathcal{B}(\mathfrak{F})}$. Since, $f\in \mathcal{B}(\mathfrak{F})$, we have
$$F_+^{-1}f\in H^2_\mathfrak{X}(\mathbb{C}_+)~\mbox{and}~F_-^{-1}f\in H^2_\mathfrak{X}(\mathbb{C}_+)^\perp.$$
Now, 
$$(E_+F_+)^{-1}(E_+f)=F_+^{-1}E_+^{-1}E_+f=F_+^{-1}f\in H^2_\mathfrak{X}(\mathbb{C}_+),$$
and
\begin{align*}
(E_-F_-)^{-1}(E_+f)&=F_-^{-1}E_-^{-1}E_+f\\
&= E_-^{-1}F_-^{-1}E_+f\hspace{.3cm}[~\mbox{using}~ (\ref{d})~]\\
&=(E_-^{-1}E_+)(F_-^{-1}f)\in H^2_\mathfrak{X}(\mathbb{C}_+)^\perp\hspace{.3cm}[~\mbox{using}~ (\ref{c})~],
\end{align*}
where the last inclusion follows from $(\ref{Equality})$ and the facts that $E_+^{-1}E_-\in \mathcal{S}^{in}\cap\mathcal{S}_*^{in}$ and $F_-^{-1}f\in H^2_\mathfrak{X}(\mathbb{C}_+)^\perp$. The norm equality is easy to verify. Hence, $E_+\mathcal{B}(\mathfrak{F})$ is isometrically contained in $\mathcal{B}(\mathfrak{E}\mathfrak{F})$. Similarly, it can be proved that $F_-\mathcal{B}(\mathfrak{E})$ is  isometrically contained in $\mathcal{B}(\mathfrak{E}\mathfrak{F})$. 
Also, the following calculation shows that $E_+(z)K_w^\mathfrak{F}(z)E_+(w)^*$ is the reproducing kernel of the Hilbert space $E_+\mathcal{B}(\mathfrak{F})$. For any $f\in\mathcal{B}(\mathfrak{F})$ and $u\in\mathfrak{X}$,
\begin{align*}
\langle E_+f,E_+K_w^\mathfrak{F}E_+(w)^*u\rangle_{E_+\mathcal{B}(\mathfrak{F})}& = \langle E_+f,E_+K_w^\mathfrak{F}E_+(w)^*u\rangle_{\mathcal{B}(\mathfrak{E}\mathfrak{F})}\\
& = \int_{-\infty}^\infty \langle F_+^{-1}(t)f(t),F_+^{-1}(t)K_w^\mathfrak{F}(t)E_+(w)^*u\rangle_\mathfrak{X} dt\\
& = \langle f,K_w^\mathfrak{F}E_+(w)^*u\rangle_{\mathcal{B}(\mathfrak{F})}\\
& = \langle f(w),E_+(w)^*u\rangle_\mathfrak{X}\\
& = \langle (E_+f)(w),u\rangle_\mathfrak{X}.
\end{align*}
Similarly, it can be shown that $F_-(z)K_w^\mathfrak{E}(z)F_-(w)^*$ is the reproducing kernel of the Hilbert space $F_-\mathcal{B}(\mathfrak{E})$. The rest of the proof follows from a general complementation theory in Hilbert spaces, which can be found in \cite{Alpay} (Chapter $1.5$).
\end{proof}
Now, we mention a particular case of the Theorem \ref{AT}, which will be used later in this section.
\begin{thm}
\label{Factorization}
Let $A(z)$ be a $\Phi{(\mathfrak{X}})$-valued entire function that is invertible at least at one point. Then a factorization of $A(z)$ of the form $A(z)=N(z)A_0(z)$ holds, where $A_0(z)$ is a $\Phi{(\mathfrak{X}})$-valued entire function that is invertible for all real $z$ and $N(z)$ is a $B(\mathfrak{X})$-valued entire function of the form $(\ref{P1})$. Also, $N(z)$ is invertible for all $z$ except for those $z\in\mathbb{R}$ where $A(z)$ is not invertible.
\end{thm} 
\begin{proof}
The Fredholm analytic theorem and the fact that $A(z)$ is invertible at least at one point implies that $A(z)$ is invertible for all $z\in\mathbb{C}$ except for a discrete set. Without loss of generality, we may assume that $A(z)$ is invertible at the origin and $D_0$ is the collection of all real points, where $A(z)$ is not invertible. If $D_0$ is empty, then the factorization result follows by choosing $N(z)=I$ and $A_0(z)=A(z)$. Otherwise, let $x_1$ be an element in $D_0$ nearest to the origin and $P_1$ be the orthogonal projection operator on $(\mathrm{rng}A(x_1))^\perp $. $P_1$ is a finite rank operator, as $A(x_1)$ is a Fredholm operator. Then
\begin{equation}
 \left[I-\frac{z}{x_1}P_1\right]^{-1}=\left[I-\frac{z}{z-x_1}P_1\right]
 \end{equation}
 is an operator valued analytic function for all $z\in\mathbb{C}$ except at $x_1$, and $P_1A(x_1)=0$ implies that
 \begin{equation}
 \left[I-\frac{z}{x_1}P_1\right]^{-1}A(z)=(z-x_1-zP_1)\left[\frac{A(z)-A(x_1)}{z-x_1}\right]+A(x_1)
 \end{equation} 
 is an operator valued entire function. Thus $A(z)$ has the factorization $A(z)=N^{(1)}(z)A_0^{(2)}(z)$, where $N^{(1)}(z)=\left[I-\frac{z}{x_1}P_1\right]\exp\left[\frac{z}{x_1}P_1\right]$ and
 \begin{equation}
    A_0^{(2)}(z):= \left\{
    \begin{array}{ll}
        \exp[-\frac{z}{x_1}P_1][(z-x_1-zP_1)[\frac{A(z)-A(x_1)}{z-x_1}]+A(x_1)]  , z \neq x_1 \\
         \exp(-P_1)[A(x_1)-x_1P_1A'(x_1)]  , z = x_1.
    \end{array} \right.
\end{equation} 
Observe that $A_0^{(2)}(z)$ is a $\Phi{(\mathfrak{X}})$-valued entire function invertible at the origin. If $A_0^{(2)}(z)$ is invertible for all real $z$, the proof is complete, and $N(z)=N^{(1)}(z)$, $A_0(z)=A_0^{(2)}(z)$. Otherwise, let $x_2$ be an element in $D_0$ nearest to the origin such that $A_0^{(2)}(x_2)$ is not invertible and continue inductively. At the n-th step of the induction process, we have $A_0^{(n)}(z)$ is a $\Phi{(\mathfrak{X}})$-valued entire function invertible at the origin. Suppose $x_n$ is an element in $D_0$ nearest to the origin such that $A_0^{(n)}(x_n)$ is not invertible. Let $P_n$ is the orthogonal projection on $(\mathrm{rng}A_0^{(n)}(x_n))^\perp$. Again we have the factorization $A(z)=N^{(n)}(z)A_0^{(n+1)}(z)$, where
$$N^{(n)}(z)=\left[I-\frac{z}{x_1}P_1\right]\exp\left[\frac{z}{x_1}P_1\right]\ldots
 \left[I-\frac{z}{x_n}P_n\right]\exp\left[\frac{z}{x_n}P_n+\ldots+\frac{1}{n}\frac{z^n}{x_n^n}P_n\right]$$
 and $A_0^{(n+1)}(z)$ is a $\Phi{(\mathfrak{X}})$-valued entire function invertible at the origin. If $A_0^{(n+1)}(z)$ is invertible for all real $z$, we conclude the proof with $N(z)=N^{(n)}(z)$ and $A_0(z)=A_0^{(n+1)}(z)$. Otherwise, we keep the process moving. This discussion already covered the case when $D_0$ is finite. Now, suppose $D_0$ is infinite with $|x_n|\to\infty$ as $n\to\infty$.\\
 Now, we can apply Theorem \ref{product} to conclude that the sequence  $\{N^{(n)}(z)\}$ converges to some $B(\mathfrak{X})$-valued entire  function $N(z)$ uniformly in any bounded set with respect to the operator norm such that $I-N(z)$ is compact for all $z\in\mathbb{C}$ and invertible for all $z\in\mathbb{C}\setminus \{x_n\}$. Also, the proof of the fact that the sequence $\{A_0^{(n+1)}(z)\}$ converges to some $B(\mathfrak{X})$-valued entire function $A_0(z)$, invertible for all real $z$, uniformly in any bounded set with respect to operator norm follows from Theorem \ref{AT}. Observe that for every $z\in\mathbb{C}$ both $A(z)$ and $N(z)$ are Fredholm operators. Thus $A_0(z)$ is a $\Phi{(\mathfrak{X}})$-valued entire function follows from Proposition \ref{FP}.
 \end{proof}
\begin{rmk}
If the finite rank orthogonal projection operators $P_1, P_2,\ldots, P_n,\ldots$ in the previous theorem are pairwise commutative, then for all $x\in\mathbb{R}$, $N^{(n)}(x)^*=N^{(n)}(x)$. Since the self-adjoint operators are closed subset of $B(\mathfrak{X})$ and $N^{(n)}(x)\to N(x)$ in operator norm, $N(x)$ is self-adjoint for all $x\in\mathbb{R}$.
\end{rmk}
The next theorem establishes a connection between the factorization of Fredholm operator valued entire functions and the structure of vector valued de Branges spaces. Problem $44$ from \cite{Branges 4} states that if $\mathcal{H}(E)$ is a given de Branges space of scalar valued entire functions corresponding to a Hermite-Biehler function $E(z)$, then $E(z)=S(z)E^0(z)$, where $\mathcal{H}(E^0)$ exists, $E^0(z)$ has no real zeros, and the zeros of $S(z)$ are real zeros of $E(z)$. Moreover, the equality $\mathcal{H}(E)=S\mathcal{H}(E^0)$ holds. The following theorem generalizes this problem to the setting of vector valued de Branges spaces.
\begin{thm}
\label{Embedding}
Let $\mathfrak{E}(z)=(E_-(z),E_+(z))$ be a de Branges operator satisfying Hypothesis $\ref{H1}$. Then $E_+(z)=N(z)~E_+^0(z)$ and $E_-(z)=N(z)~E_-^0(z)$, where $N(z)$ is a $B(\mathfrak{X})$-valued entire function of the form $(\ref{P1})$ and $\mathfrak{E}^0(z)=(E_-^0(z),E_+^0(z))$ is a de Branges operator such that 
\begin{enumerate}
\item $E_\pm^0(z)$ are invertible for all $z\in\mathbb{R}$, and
\item The equality $\mathcal{B}(\mathfrak{E})=N\mathcal{B}(\mathfrak{E}^0)$ holds.
\end{enumerate}
\end{thm}
\begin{proof}
Since $\mathcal{B}(\mathfrak{E})$ is a de Branges space, Lemma \ref{Pole} implies that the real points where both $E_+$ and $E_-$ are not invertible are the same, and we denote the collection as $D$. Now, we want the factorization of $E_+$ and $E_-$ as in the Theorem \ref{Factorization}. Without loss of generality, we may assume that  $E_+$ and $E_-$ are invertible at the origin. Suppose $x_1$ is an element in $D$ nearest to the origin, and $P_1$ is the orthogonal projection operator on $(\mathrm{rng}E_+(x_1))^\perp=(\mathrm{rng}E_-(x_1))^\perp$. Then $E_+(z)=N^{(1)}(z)S_2(z)$ and $E_-(z)=N^{(1)}(z)T_2(z)$, where $N^{(1)}(z)$ as in the previous theorem, 
\begin{equation}
    S_2(z):= \left\{
    \begin{array}{ll}
        \exp[-\frac{z}{x_1}P_1][((z-x_1)I-zP_1)[\frac{E_+(z)-E_+(x_1)}{z-x_1}]+E_+(x_1)]  , z \neq x_1 \\
         \exp(-P_1)[E_+(x_1)-x_1P_1E_+'(x_1)]  , z = x_1
    \end{array} \right.
\end{equation}
and
\begin{equation}
    T_2(z):= \left\{
    \begin{array}{ll}
        \exp[-\frac{z}{x_1}P_1][((z-x_1)I-zP_1)[\frac{E_-(z)-E_-(x_1)}{z-x_1}]+E_-(x_1)]  , z \neq x_1 \\
         \exp(-P_1)[E_-(x_1)-x_1P_1E_-'(x_1)]  , z = x_1.
    \end{array} \right.
\end{equation}
Here both $S_2(z)$ and $T_2(z)$ are $\Phi{(\mathfrak{X}})$-valued entire functions and $S_2(0)=E_+(0)$, $T_2(0)=E_-(0)$. Also, from the factorizations of $E_+(z)$ and $E_-(z)$, it is clear that $S_2(z)$ is invertible for all $z\in\mathbb{C}_+$, and $T_2(z)$ is invertible for all $z\in\mathbb{C}_-$. Then for all $z\in\mathbb{C}_+$ and for almost every $z\in\mathbb{R}$, we have
\begin{align*}
E_+(z)^{-1}E_-(z)&=S_2(z)^{-1}N^{(1)}(z)^{-1}N^{(1)}(z)T_2(z)\\
&=S_2(z)^{-1}T_2(z).
\end{align*}
This implies that $S_2$ and $T_2$ satisfy Hypothesis \ref{H1}. Thus for all $x\in\mathbb{R}$, again we have $\mathrm{rng}S_2(x)=\mathrm{rng}T_2(x)$. Now, as in the previous theorem, we continue factoring $E_+(z)$ and $ E_-(z)$ inductively. Observe that in every inductive step, the first factors of $E_+(z)$ and $E_-(z)$ are the same, and the second factors satisfy hypothesis \ref{H1}. Finally, whether $D$ is finite or infinite, we have the factorization of $E_+$ and $E_-$ as
\begin{equation}
\label{dB operator factorization}
E_+(z)=N(z)E_+^0(z),\hspace{.5cm}E_-(z)=N(z)E_-^0(z),
\end{equation}
where $N(z)$ as in the previous theorem and $E_+^0$, $E_-^0$ satisfy hypothesis \ref{H1}. Indeed, $E_+(z)$ is invertible for all $z\in\mathbb{C}_+$, and $N(z)$ is invertible for all $z\not\in D\subset\mathbb{R}$. Thus $E_+^0(z)$ is invertible for all $z\in\mathbb{C}_+$. Similarly, $E_-^0(z)$ is invertible for all $z\in\mathbb{C}_-$. This implies $\mathcal{B}(\mathfrak{E}^0)$ exists with the de Branges operator $\mathfrak{E}^0(z)=(E_-^0(z),E_+^0(z))$. Now, $(\ref{dB operator factorization})$ and Lemma \ref{Isometry} together imply that $N\mathcal{B}(\mathfrak{E}^0)$ is isometrically contained in $\mathcal{B}(\mathfrak{E})$. Moreover, since the zeros of $N(z)$ are only the real zeros of $E_+(z)$ (same as $E_-(z)$), given any $f\in\mathcal{B}(\mathfrak{E})$, $N(z)^{-1}f(z)\in\mathcal{B}(\mathfrak{E}^0)$. Hence, the equality $\mathcal{B}(\mathfrak{E})=N\mathcal{B}(\mathfrak{E}^0)$ holds.
\end{proof}
\begin{rmk}
Observe that the two de Branges spaces $\mathcal{B}(\mathfrak{E})$ and $\mathcal{B}(\mathfrak{E}^0)$ involved in the previous theorem are in the sense of our present consideration, i.e., the components of the corresponding de Branges operators are in the class of Fredholm operator valued entire functions which need not be of the form considered in \cite{Rovnyak}.
\end{rmk}

\section{Associated functions and multiplication operator in de Branges spaces}
\label{Sec-6}
This section deals with the $B(\mathfrak{X})$-valued entire functions associated with a de Branges space $\mathcal{B}(\mathfrak{E})$, where $\mathfrak{E}(z)=(E_-(z),E_+(z))$ is the de Branges operator. The relationship between these associated functions and the multiplication operator $\mathfrak{T}$ is also explored. It is worth noting that a discussion of associated functions for scalar valued de Branges spaces of entire functions can be found in \cite{Branges 4} (second chapter, section $25$). A similar discussion employing a different approach within the framework of de Branges spaces of entire functions with matrix valued reproducing kernels can also be found in \cite{ArD08} and \cite{ArD18}. Moreover, \cite{Branges 5} examines operator valued associated functions $S(z)$, where $I-S(z)$ are compact operators for all $z\in\mathbb{C}$. Recall that a $B(\mathfrak{X})$-valued entire function $S(z)$ is said to be associated with the de Branges space $\mathcal{B}(\mathfrak{E})$ if $S(\alpha)$ is invertible for some $\alpha\in\mathbb{C}$ and for every $f\in\mathcal{B}(\mathfrak{E})$,
$$\frac{f(z)-S(z)S(\alpha)^{-1}f(\alpha)}{z-\alpha}\in\mathcal{B}(\mathfrak{E}).$$
\begin{rmk}
If $\alpha\in\mathbb{C}$ is such that $K_\alpha(\alpha)$ is invertible, then $S(z)=K_\alpha(z)$ is associated with $\mathcal{B}(\mathfrak{E})$ 
\begin{enumerate}
\item for all $\alpha\in\mathbb{C}\setminus \mathbb{R}$, and
\item for all $\alpha\in\mathbb{R}$, where $E_+(\alpha)$ is invertible.
\end{enumerate}
\end{rmk}
\begin{rmk}
Observe that $S(z)=I$ is associated with $\mathcal{B}(\mathfrak{E})$ if and only if $\mathcal{B}(\mathfrak{E})$ is invariant under the generalized backward shift operator $R_\alpha$ for every $\alpha\in\mathbb{C}$.
\end{rmk}
The following theorem gives a sufficient condition for a $B(\mathfrak{X})$-valued entire function to be associated with the de Branges space $\mathcal{B}(\mathfrak{E})$.
\begin{thm}
\label{AFS}
Let $\mathcal{B}(\mathfrak{E})$ be a vector valued de Branges space corresponding to the de Branges operator $\mathfrak{E}(z)=(E_-(z),E_+(z))$ and $S(z)$ be a $B(\mathfrak{X})$-valued entire function such that 
\begin{equation}
\label{Associated function sufficient}
\frac{E_+^{-1}S}{\rho_i}u\in H^2_{\mathfrak{X}}(\mathbb{C}_+)~~\mbox{and}~~\frac{E_-^{-1}S}{\rho_{-i}}u\in H^2_{\mathfrak{X}}(\mathbb{C}_+)^\perp \hspace{.3cm}\mbox{for all}~u\in\mathfrak{X}.
\end{equation}
 Then 
\begin{enumerate}
\item $E_+^{-1}S$ is analytic in $\overline{\mathbb{C}_+}$.
\item $E_-^{-1}S$ is analytic in $\overline{\mathbb{C}_-}$.
\item Moreover, if $S(\alpha)$ is invertible for some $\alpha\in\mathbb{C}$, the linear transformation $R_S(\alpha):\mathcal{B}(\mathfrak{E})\to \mathcal{B}(\mathfrak{E})$ defined by
\begin{equation}
f(z)\mapsto \frac{f(z)-S(z)S(\alpha)^{-1}f(\alpha)}{z-\alpha},\label{Node}
\end{equation}
is everywhere defined bounded linear operator on $\mathcal{B}(\mathfrak{E})$.
\end{enumerate}
\end{thm}
\begin{proof}
Since $\frac{E_+^{-1}S}{\rho_i}u\in H^2_{\mathfrak{X}}(\mathbb{C}_+)$ for all $u\in\mathfrak{X}$, no $z\in\overline{\mathbb{C}_+}$ is a pole of $E_+^{-1}Su$. Now, due to Fredholm analytic theorem, it is clear that $E_+^{-1}Su$ is analytic in $\overline{\mathbb{C}_+}$ for all $u\in\mathfrak{X}$. Equivalently, $E_+^{-1}S$ is analytic in $\overline{\mathbb{C}_+}$. Similarly, it can be proved that $E_-^{-1}S$ is analytic in $\overline{\mathbb{C}_-}$.\\
Now, it remains to prove $(3)$. Suppose $f\in \mathcal{B}(\mathfrak{E})$ and $\alpha\in\mathbb{C}$ is such that $S(\alpha)^{-1}\in B(\mathfrak{X})$. Assume that
 $$g(z)=E_+^{-1}(z)[f(z)-S(z)S(\alpha)^{-1}f(\alpha)]$$ and $$h(z)=E_-^{-1}(z)[f(z)-S(z)S(\alpha)^{-1}f(\alpha)].$$
Since $f\in\mathcal{B}(\mathfrak{E})$, we have $E_+^{-1}f\in H^2_\mathfrak{X}(\mathbb{C}_+)$, which implies $\frac{E_+^{-1}f}{\rho_i}\in H^2_\mathfrak{X}(\mathbb{C}_+)$. Due to (\ref{Associated function sufficient}), it is also true that  $\frac{E_+^{-1}S}{\rho_i}S(\alpha)^{-1}f(\alpha)\in H^2_\mathfrak{X}(\mathbb{C}_+)$. Thus, we have $\frac{g}{\rho_i}\in H^2_\mathfrak{X}(\mathbb{C}_+)$. Similarly, it can be proved that $\frac{h}{\rho_{-i}}\in H^2_\mathfrak{X}(\mathbb{C}_+)^\perp$.\\
If $\alpha\in\overline{\mathbb{C}_+}$, then 
$$E_+^{-1}\frac{f(z)-S(z)S(\alpha)^{-1}f(\alpha)}{z-\alpha}=R_\alpha g\in H^2_\mathfrak{X}(\mathbb{C}_+),$$
as $\frac{g}{\rho_i}\in H^2_\mathfrak{X}(\mathbb{C}_+)$. If $\alpha\in\mathbb{C}_-$, then
$$E_+^{-1}\frac{f(z)-S(z)S(\alpha)^{-1}f(\alpha)}{z-\alpha}=\frac{\rho_i(z)}{z-\alpha}E_+^{-1}(z)\frac{f(z)-S(z)S(\alpha)^{-1}f(\alpha)}{\rho_i(z)}$$
belongs to $H^2_\mathfrak{X}(\mathbb{C}_+)$ as $\frac{\rho_i(z)}{z-\alpha}\in H^\infty$. Similarly, it can be proved that, for all $\alpha\in\mathbb{C}$ and $f\in \mathcal{B}(\mathfrak{E})$,
$$E_-^{-1}\frac{f(z)-S(z)S(\alpha)^{-1}f(\alpha)}{z-\alpha}\in H^2_\mathfrak{X}(\mathbb{C}_+)^\perp.$$
Thus, it is clear that $R_S(\alpha)$ is an everywhere defined linear transformation on $\mathcal{B}(\mathfrak{E})$. The boundedness of $R_S(\alpha)$ can be proved by the closed graph theorem.
\end{proof}
Given a $B(\mathfrak{X})$-valued entire function $S(z)$, we denote $\rho_S=\{z\in\mathbb{C}:S(z)^{-1}\in B(\mathfrak{X})\}$. Then, the preceding theorem implies that if $S(z)$ is an associated function of $\mathcal{B}(\mathfrak{E})$, $R_S(z)$ is a bounded linear operator for all $z\in\rho_S$. Also it is satisfying the following resolvent identity
\begin{equation}
\label{Resolvent identity}
R_S(\alpha)-R_S(\beta)=(\alpha-\beta)R_S(\alpha)R_S(\beta)\hspace{.3cm}\mbox{for any}~\alpha, \beta\in\rho_S.
\end{equation}
Note that if the associated function $S(z)$ is $\Phi(\mathfrak{X})$-valued entire function, then $\rho_S=\mathbb{C}\setminus D$, where $D$ is a discrete set. For example, we can consider $E_+(z)$ and $E_-(z)$. 
The next theorem gives a converse result of the previous theorem.
\begin{thm}
Let $\mathcal{B}(\mathfrak{E})$ be a vector valued de Branges space corresponding to the de Branges operator $\mathfrak{E}(z)=(E_-(z),E_+(z))$ and $S(z)$ be a $B(\mathfrak{X})$-valued entire function. Suppose $K_\alpha(\alpha)$ is invertible for some number $\alpha\in\mathbb{C}$ and for every $f\in \mathcal{B}(\mathfrak{E})$ 
$$\frac{f(z)-S(z)S(\alpha)^{-1}f(\alpha)}{z-\alpha}\in\mathcal{B}(\mathfrak{E}).$$
Then $(\ref{Associated function sufficient})$ holds.
\end{thm}
\begin{proof}
Suppose $f\in \mathcal{B}(\mathfrak{E})$, then $E_+^{-1}f\in H^2_\mathfrak{X}(\mathbb{C}_+)$ and $E_-^{-1}f\in H^2_\mathfrak{X}(\mathbb{C}_+)^\perp$, which enforce that $\frac{E_+^{-1}f}{\rho_i}\in H^2_\mathfrak{X}(\mathbb{C}_+)$ and $\frac{E_-^{-1}f}{\rho_{-i}}\in H^2_\mathfrak{X}(\mathbb{C}_+)^\perp$. Also $\frac{f(z)-S(z)S(\alpha)^{-1}f(\alpha)}{z-\alpha}\in\mathcal{B}(\mathfrak{E})$ implies that $E_+^{-1}\frac{f(z)-S(z)S(\alpha)^{-1}f(\alpha)}{z-\alpha}\in H^2_\mathfrak{X}(\mathbb{C}_+)$ and $E_-^{-1}\frac{f(z)-S(z)S(\alpha)^{-1}f(\alpha)}{z-\alpha}\in H^2_\mathfrak{X}(\mathbb{C}_+)^\perp$. Now, assume $g(z)=E_+^{-1}(z)[f(z)-S(z)S(\alpha)^{-1}f(\alpha)]$ and $h(z)=E_-^{-1}(z)[f(z)-S(z)S(\alpha)^{-1}f(\alpha)]$. Thus
\begin{align*}
\frac{g(z)}{\rho_i(z)}&=\frac{E_+^{-1}(z)[f(z)-S(z)S(\alpha)^{-1}f(\alpha)]}{\rho_i(z)}\\
&= \frac{(z-\alpha)}{\rho_i(z)}\frac{E_+^{-1}(z)[f(z)-S(z)S(\alpha)^{-1}f(\alpha)]}{z-\alpha}.
\end{align*}
This implies $\frac{g}{\rho_i}\in H^2_\mathfrak{X}(\mathbb{C}_+)$ as $\frac{z-\alpha}{\rho_i(z)}\in H^\infty$ and $E_+^{-1}\frac{f(z)-S(z)S(\alpha)^{-1}f(\alpha)}{z-\alpha}\in H^2_\mathfrak{X}(\mathbb{C}_+)$. Now, we have  $\frac{g}{\rho_i},\frac{E_+^{-1}f}{\rho_i}\in H^2_\mathfrak{X}(\mathbb{C}_+)$. Thus $\frac{E_+^{-1}S}{\rho_i}S(\alpha)^{-1}f(\alpha)\in H^2_\mathfrak{X}(\mathbb{C}_+)$ for all $f\in \mathcal{B}(\mathfrak{E})$. Since $\{f(\alpha):f\in \mathcal{B}(\mathfrak{E})\}=\mathrm{rng}K_\alpha(\alpha)$ and $K_\alpha(\alpha)$ is invertible, $\frac{E_+^{-1}S}{\rho_i}u\in H^2_{\mathfrak{X}}(\mathbb{C}_+)$ for all $u\in\mathfrak{X}$. Similarly, it can be proved that $\frac{E_-^{-1}S}{\rho_{-i}}S(\alpha)^{-1}f(\alpha)\in H^2_\mathfrak{X}(\mathbb{C}_+)^\perp$ for all $f\in \mathcal{B}(\mathfrak{E})$. Thus $\frac{E_-^{-1}S}{\rho_{-i}}u\in H^2_{\mathfrak{X}}(\mathbb{C}_+)^\perp$ for all $u\in\mathfrak{X}$.
\end{proof}
At this point, we can conclude a result regarding the invariance of vector valued de Branges spaces under the generalized backward shift operator. We write the result as a theorem below, whose proof is a particular case of the previous two theorems.
\begin{thm}
Suppose $\mathcal{B}(\mathfrak{E})$ is a vector valued de Branges space corresponding to the de Branges operator $\mathfrak{E}(z)=(E_-(z),E_+(z))$. Then $\mathcal{B}(\mathfrak{E})$ is invariant under the generalized backward shift operator $R_\alpha$ for all $\alpha\in\mathbb{C}$ if
$$\frac{E_+^{-1}}{\rho_i}u\in H^2_{\mathfrak{X}}(\mathbb{C}_+)\hspace{.3cm}\mbox{and}\hspace{.3cm}\frac{E_-^{-1}}{\rho_{-i}}u\in H^2_{\mathfrak{X}}(\mathbb{C}_+)^\perp \hspace{.5cm}\mbox{for all}~u\in\mathfrak{X}.$$
Conversely, if there exists a number $\alpha\in\mathbb{C}$ such that $K_\alpha(\alpha)$ is inertible and $R_\alpha \mathcal{B}(\mathfrak{E})\subseteq \mathcal{B}(\mathfrak{E})$, then
$$\frac{E_+^{-1}}{\rho_i}u\in H^2_{\mathfrak{X}}(\mathbb{C}_+)\hspace{.3cm}\mbox{and}\hspace{.3cm}\frac{E_-^{-1}}{\rho_{-i}}u\in H^2_{\mathfrak{X}}(\mathbb{C}_+)^\perp\hspace{.5cm} \mbox{for all}~u\in\mathfrak{X}.$$
\end{thm}
Now, the following proposition connects the associated functions of $\mathcal{B}(\mathfrak{E})$ and the multiplication operator $\mathfrak{T}$ in terms of closed linear relations on $\mathcal{B}(\mathfrak{E})$. A linear relation from $\mathcal{B}(\mathfrak{E})$  to $\mathcal{B}(\mathfrak{E})$  is nothing but a linear subspace of $\mathcal{B}(\mathfrak{E})\times\mathcal{B}(\mathfrak{E})$. For a more general discussion about linear relations, refer to \cite{Behrndt} and their connection to de Branges spaces with matrix valued reproducing kernels, see \cite{ArD20}.
\begin{prop}
Let $\mathcal{B}(\mathfrak{E})$ be a vector valued de Branges space corresponding to the de Branges operator $\mathfrak{E}(z)=(E_-(z),E_+(z))$ and $S(z)$ is a $\Phi(\mathfrak{X})$-valued entire function associated to it. Then there exists a closed linear relation $T$ on $\mathcal{B}(\mathfrak{E})$ such that $\mathfrak{T}\subseteq T $ and the following identity holds
$$(T-\alpha I)^{-1}f(z)=R_S(\alpha)f(z)\hspace{.3cm}\mbox{for all}~f\in \mathcal{B}(\mathfrak{E})~\mbox{and}~\alpha\in\rho_S.$$
\end{prop}
\begin{proof}
Since $S(z)$ is a $\Phi(\mathfrak{X})$-valued entire function and associated with $\mathcal{B}(\mathfrak{E})$, there exists a discrete set $D\subset \mathbb{C}$ such that $\rho_S=\mathbb{C}\setminus D$. Also, $R_S(\alpha)$ is a bounded linear operator on $\mathcal{B}(\mathfrak{E})$ for all $\alpha\in\rho_S$ and satisfies the resolvent identity $(\ref{Resolvent identity})$ for all $\alpha, \beta\in\rho_S$. Then, due to Proposition $1.2.9$ in \cite{Behrndt}, there exists a closed linear relation $T$ in $\mathcal{B}(\mathfrak{E})$ such that $\rho_S\subseteq \rho(T)$ and $(T-\alpha I)^{-1}=R_S(\alpha)$ for all $\alpha\in\rho_S$. Now, to show that $\mathfrak{T}\subseteq T $, we only need to show that $(\mathfrak{T}-\alpha I)^{-1}\subseteq R_S(\alpha)$. Observe that if $f\in\mathrm{rng}(\mathfrak{T}-\alpha I)$, then $f(\alpha)=0$. Thus
$$R_S(\alpha)f(z)=\frac{f(z)}{z-\alpha}=(\mathfrak{T}-\alpha I)^{-1}f(z).$$
\end{proof}
\begin{rmk}
In the last proposition, the closed linear relation $T$ is proper if and only if a linear manifold $M$ exists in $\mathfrak{X}$ such that $S(z)u\in \mathcal{B}(\mathfrak{E})$ for all $u\in M$. This can easily be proved by using Lemma $1.1.6$ in \cite{Behrndt}.
\end{rmk}
Note that the connection between associated functions of de Branges-Pontryagin spaces of $\mathbb{C}^n$-valued entire functions and extensions of the multiplication operator is explicitly discussed in Proposition $8.2$ of \cite{Derkach1} and Proposition $2.6$ of \cite{Derkach2}. This leads us to find a description of the set of all associated functions in terms of $E_-$ and $E_+$. In this direction, we present the following observation. However, a complete characterization of this set remains open.
\begin{prop}
Let $\mathcal{B}(\mathfrak{E})$ be a vector valued de Branges space corresponding to the de Branges operator $\mathfrak{E}(z)=(E_-(z),E_+(z))$, and suppose there exists $\alpha\in\mathbb{C}_+$ such that $K_\alpha(\alpha)$ is invertible. Then, for any unitary operator $V$ on $\mathfrak{X}$, the $B(\mathfrak{X})$-valued entire function $$S(z)=E_-(z)+E_+(z)V$$ is an associated function of $\mathcal{B}(\mathfrak{E})$.
\end{prop}
\begin{proof}
If $\alpha\in\mathbb{C}_+$ is such that $K_\alpha(\alpha)$ is invertible, then $||E_+^{-1}(\alpha)E_-(\alpha)||<1$. Consequently, $S(\alpha)$ is invertible. The rest of the proof then follows from Theorem \ref{AFS}. In fact, for any $u\in\mathfrak{X}$, we have
$$\frac{E_+^{-1}Su}{\rho_i}=\frac{E_+^{-1}(E_-+E_+V)u}{\rho_i}=E_+^{-1}E_-\frac{u}{\rho_i}+\frac{Vu}{\rho_i}\in H^2_{\mathfrak{X}}(\mathbb{C}_+),$$
and
$$\frac{E_-^{-1}Su}{\rho_{-i}}=\frac{E_-^{-1}(E_-+E_+V)u}{\rho_{-i}}=\frac{u}{\rho_{-i}}+E_-^{-1}E_+\frac{Vu}{\rho_{-i}}\in H^2_{\mathfrak{X}}(\mathbb{C}_+)^\perp.$$
\end{proof}
We conclude this section with the following proposition, which observes that if $S(z)$ is an associated function of $\mathcal{B}(\mathfrak{E}^0)$, then $N(z)S(z)$ is an associated function of $\mathcal{B}(\mathfrak{E})$. The proof of this proposition follows from Theorem \ref{Embedding} and Theorem \ref{AFS}.
\begin{prop}
Let $\mathcal{B}(\mathfrak{E})$, $\mathcal{B}(\mathfrak{E}^0)$, and $N(z)$ be as in Theorem \ref{Embedding}. $S(z)$ is a $B(\mathfrak{X})$-valued entire function such that 
$$\frac{(E_+^0)^{-1}S}{\rho_i}u\in H^2_{\mathfrak{X}}(\mathbb{C}_+)~~\mbox{and}~~\frac{(E_-^0)^{-1}S}{\rho_{-i}}u\in H^2_{\mathfrak{X}}(\mathbb{C}_+)^\perp \hspace{.3cm}\mbox{for all}~u\in\mathfrak{X}.$$
Then the following also holds
$$\frac{E_+^{-1}NS}{\rho_i}u\in H^2_{\mathfrak{X}}(\mathbb{C}_+)~~\mbox{and}~~\frac{E_-^{-1}NS}{\rho_{-i}}u\in H^2_{\mathfrak{X}}(\mathbb{C}_+)^\perp \hspace{.3cm}\mbox{for all}~u\in\mathfrak{X}.$$
\end{prop}
\section{Connection with the operator nodes}
\label{Sec-7}
In this section, we recall the idea of operator nodes and establish a connection with the vector valued de Branges spaces $\mathcal{B}(\mathfrak{E})$. A comprehensive study of operator nodes can be found in \cite{Brodskii}, and the connection with the de Branges spaces with matrix valued reproducing kernels can be found in \cite{ArD08}. Suppose $\mathfrak{H}$ and $\mathfrak{G}$ are two separable Hilbert spaces and $A\in B(\mathfrak{H})$, $T\in B(\mathfrak{H},\mathfrak{G})$, and $J$ is a signature operator in $B(\mathfrak{G})$, i.e., $$J=J^*=J^{-1}.$$
Then, the set of these Hilbert spaces and operators is called an operator node if
$$A-A^*=iT^*JT$$ 
and is denoted as 
$$\Theta=
\begin{pmatrix}
    A       & T & J  \\
    \mathfrak{H}       &  & \mathfrak{G}
\end{pmatrix}.
$$
Here, $\mathfrak{H}$ is called interior space, and $\mathfrak{G}$ is called exterior space. Also, $A$ is the basic operator, $T$ is the canal operator, and $J$ is the directing operator. The operator node $\Theta$ is called simple if $$\cap_{n=0}^\infty \ker(TA^n)=\{0\}$$
and is called dissipative if $J=I_\mathfrak{G}$, the identity operator on $\mathfrak{G}$. The characteristic function of the operator node $\Theta$ is the operator valued function
$$W_\Theta(z)=I_\mathfrak{G}+izT(I_\mathfrak{H}-zA)^{-1}T^*J\hspace{.3cm}\mbox{for all}~z\in Z_A.$$
\begin{thm}
\label{Operator node connection}
Let $\mathfrak{E}(z)=(E_-(z),E_+(z))$ be a de Branges operator such that $E_+(0)=E_-(0)=I$ and $F=E_+^{-1}E_-$. Then, $F$ is the characteristic operator function of the simple, dissipative operator node 
$$\Theta=
\begin{pmatrix}
    R_{E_+}(0)      & T & I  \\
       \mathcal{B}(\mathfrak{E})   &  & \mathfrak{X}
\end{pmatrix},
$$
where $T$ is defined by $Tf=\sqrt{2\pi}f(0)$ for all $f\in\mathcal{B}(\mathfrak{E})$.
\end{thm}
\begin{proof}
Due to $(\ref{Node})$, it is clear that $R_{E_+}(0)$ is a bounded operator on $\mathcal{B}(\mathfrak{E})$, and since $\mathcal{B}(\mathfrak{E})$ is a RKHS, $T$ is also a bounded operator from $\mathcal{B}(\mathfrak{E})$ to $\mathfrak{X}$. Now, for every $f,g\in\mathcal{B}(\mathfrak{E})$, using Cauchy integral formula, we can show that
\begin{align*}
\langle R_{E_+}(0)f,g\rangle_{\mathcal{B}(\mathfrak{E})}&-\langle f,R_{E_+}(0)g\rangle_{\mathcal{B}(\mathfrak{E})}\\
&=\int_{-\infty}^\infty\langle E_+^{-1}(t)\frac{f(t)-E_+(t)f(0)}{t},E_+^{-1}(t)g(t)\rangle_\mathfrak{X}~dt\\
&-\int_{-\infty}^\infty\langle E_+^{-1}(t)f(t),E_+^{-1}(t)\frac{g(t)-E_+(t)g(0)}{t}\rangle_\mathfrak{X}~dt\\
&=i\pi\langle f(0),g(0)\rangle_\mathfrak{X}+i\pi\langle f(0),g(0)\rangle_\mathfrak{X}\\
&=2i\pi\langle f(0),g(0)\rangle_\mathfrak{X}\\
&=\langle iT^*Tf,g\rangle_{\mathcal{B}(\mathfrak{E})}.
\end{align*}
Thus, $R_{E_+}(0)-(R_{E_+}(0))^*= iT^*T$, which implies $\Theta$ is an operator node. Now, suppose $f\in\cap_{n=0}^\infty\ker(T(R_{E_+}(0))^n)$. Then a simple calculation shows that $f^{(n)}(0)=0$ for all $n$, which implies $\cap_{n=0}^\infty\ker( T(R_{E_+}(0))^n)=\{0\}$. Thus, $\Theta$ is a simple operator node. Now, consider $\Omega=\{z\in\mathbb{C}:E_+(z)~\mbox{is invertible}\}$. Then, for any $z\in\Omega$, we can have the following inverse
\begin{equation}
[I-zR_{E_+}(0)]^{-1}=[I+zR_{E_+}(z)].\label{inverse}
\end{equation}
Now, a straightforward calculation shows the following identity
$$W_\Theta(z)=I+izT(I-zR_{E_+}(0))^{-1}T^*=E_+^{-1}(z)E_-(z)\hspace{.3cm}\mbox{for all}~z\in\Omega.$$
\end{proof}
An operator node 
$$\Theta=
\begin{pmatrix}
    A       & T & J  \\
    \mathfrak{H}       &  & \mathfrak{G}
\end{pmatrix}.
$$
is said to be a Volterra node if the basic operator $A$ is a Volterra operator (i.e., $A$ is compact and has the spectrum $\{0\}$) and the canal operator $T$ is compact. Given an operator $A\in B(\mathfrak{X})$, recall that the real and imaginary parts of this operator are respectively
$$\frac{A+A^*}{2}\hspace{.3cm}\mbox{and}\hspace{.3cm}\frac{A-A^*}{2i}.$$
Now, we mention a corollary of the previous theorem, which deals with the case when the operator node connected with $\mathcal{B}(\mathfrak{E})$ is a Volterra node.
\begin{cor}
If in the setting of Theorem \ref{Operator node connection}, we assume that $E_+(z)$ is invertible for all $z\in\mathbb{C}$, and $K_0(0)$ is a compact operator, then the operator node
$$\Theta=
\begin{pmatrix}
    R_{E_+}(0)      & T & I  \\
       \mathcal{B}(\mathfrak{E})   &  & \mathfrak{X}
\end{pmatrix},
$$
is a Volterra node.
\end{cor}
\begin{proof}
We need to show that $R_{E_+}(0)$ is a Volterra operator and $T$ is a compact operator. Due to $(\ref{inverse})$, it is clear that the spectrum of $R_{E_+}(0)$ is $\{0\}$. Since $TT^*=2\pi K_0(0)$ and $K_0(0)$ is compact, $T$ is also a compact operator. Then, the imaginary part of the basic operator $R_{E_+}(0)$ is $\frac{T^*T}{2}$, which is compact. Since a bounded operator on a Hilbert space is a Volterra operator, if its spectrum contains only zero and its imaginary part is compact (see \cite{Brodskii}, Theorem $10.1$), $R_{E_+}(0)$ is a Volterra operator.
\end{proof}

\section{Appendix}
In this appendix, we provide a detailed proof of Theorem \ref{AT}.

\begin{proof}[Proof of Theorem \ref{AT}]
 Since $A(z)$ is a Fredholm operator valued entire function and $A(z_0)$ is an invertible operator, Theorem \ref{FAT} implies that there exists a discrete set of complex numbers $D=\{z_1,z_2,\ldots\}$ such that $A(z)$ is an invertible operator for all $z\in\mathbb{C}\setminus D$. Since $D$ is a discrete set, $D$ can be a finite set or an infinite set such that $|z_n|\to \infty$ as $n\to\infty$. Now, we focus on proving the first factorization of $A(z)$ that appeared in (\ref{AE}). \\
 If $D$ is an empty set, the result follows with $P(z)=I$ and $E(z)=A(z)$. Otherwise, without loss of generality, we may assume that $z_1$ is the point in $D$ nearest to $z_0$. Since $A(z_1)$ is a Fredholm operator, $\mathrm{rng}A(z_1)$ is a closed subspace of $\mathfrak{X}$, and $(\mathrm{rng}A(z_1))^\perp$ is finite dimensional. Let $P_1$ be the orthogonal projection operator on  $(\mathrm{rng}A(z_1))^\perp$. Then 
 \begin{equation}
 \label{inv}
 \left[I-\frac{z-z_0}{z_1-z_0}P_1\right]^{-1}=\left[I-\frac{z-z_0}{z-z_1}P_1\right]
 \end{equation}
 is an operator valued analytic function on $\mathbb{C}\setminus\{z_1\}$. Also, using the fact that $P_1A(z_1)=0$, we have
 \begin{equation}
 \left[I-\frac{z-z_0}{z_1-z_0}P_1\right]^{-1}A(z)=[(z-z_1)-(z-z_0)P_1]\left[\frac{A(z)-A(z_1)}{z-z_1}\right]+A(z_1)
 \end{equation}
 is an operator valued entire function. This gives $A(z)=G_1(z)~E_2(z)$, where $G_1(z)$ as in Theorem \ref{product} and 
 \begin{equation}
    E_2(z):= \left\{
    \begin{array}{ll}
        \exp[-\frac{z-z_0}{z_1-z_0}P_1][(z-z_1-(z-z_0)P_1)[\frac{A(z)-A(z_1)}{z-z_1}]+A(z_1)]  , z \neq z_1 \\
         \exp(-P_1)[A(z_1)-(z_1-z_0)P_1A'(z_1)]  , z = z_1.
    \end{array} \right.
\end{equation}
It is easy to observe that $E_2(z)$ is a Fredholm operator valued entire function, and $E_2(z_0)=A(z_0)$ is an invertible operator. If $E_2(z)$ is invertible for all $z\in\mathbb{C}$, the proof of this part is complete, and $P(z)=G_1(z)$, $E(z)=E_2(z)$. Otherwise, let $z_2\in D$ nearest to $z_0$ such that $E_2(z_2)$ is not invertible and continue inductively. Now, at the n-th phase, $E_n(z)$ is a Fredholm operator valued entire function, and $E_n(z_0)$ is an invertible operator. Suppose $z_n\in D$ nearest to $z_0$ such that $E_n(z_n)$ is not an invertible operator. Again, taking the orthogonal projection $P_n$ on $(\mathrm{rng}E_n(z_n))^\perp$, we get $A(z)=G_n(z)E_{n+1}(z)$ for all $z\in\mathbb{C}$, where $G_n(z)$ as in Theorem \ref{product} and $E_{n+1}(z)$ is a Fredholm operator valued entire function and $E_{n+1}(z_0)$ is invertible. If $E_{n+1}(z)$ is invertible for all $z\in\mathbb{C}$, we may stop the inductive process and consider $P(z)=G_n(z)$, $E(z)=E_{n+1}(z)$. Otherwise, we will move on to the next phase. If $D$ is finite, this process will stop after finite steps, and we will get the desired factorization of $A(z)$.\\
Suppose $D$ is infinite, then $|z_n|\to\infty$ as $n\to\infty$ and 
\begin{equation}
A(z)=G_n(z)E_{n+1}(z)\hspace{.5cm}\mbox{for all}~z\in\mathbb{C}~\mbox{and}~n\in\mathbb{N}.\label{induction}
\end{equation}
 Also, all $P_n$ are finite rank orthogonal projections. Thus we can apply Theorem \ref{product} to obtain an operator valued entire function $P(z)$ such that $I-P(z)$ is compact for every $z\in\mathbb{C}$ and $P(z)$ is invertible for all $z\in\mathbb{C}
\setminus D$. Now, we want to show that $\{E_n(z)\}$ is a uniformly Cauchy sequence in any bounded set with respect to the operator norm. By construction, it is clear that, for every $n\in\mathbb{N}$
$$E_n(z)=B_n\exp(\tilde{B}_n)E_{n+1}(z),$$
where $B_n$ and $\tilde{B}_n$ are as defined in the previous theorem. Thus for any $m\leq n\leq s$, we have 
\begin{multline}
||E_m(z)-E_n(z)||\leq ||E_s(z)||\exp\left(\sum_{k=m}^s\frac{|\frac{z-z_0}{z_k-z_0}|^{k+1}}{1-|\frac{z-z_0}{z_k-z_0}|}\right)\\
-||E_s(z)||\exp\left(\sum_{k=n}^s\frac{|\frac{z-z_0}{z_k-z_0}|^{k+1}}{1-|\frac{z-z_0}{z_k-z_0}|}\right).
\end{multline}
This implies that $E(z)=\lim_{n\to\infty}E_{n+1}(z)$ exists uniformly on bounded sets with respect to the operator norm. From (\ref{induction}), we conclude that $A(z)=P(z)E(z)$. Now, we only need to show that $E(z)$ is invertible for all $z\in\mathbb{C}$.\\
By construction, it is clear that $E_n(z)$ has invertible values when $|z-z_0|<|z_n-z_0|$. Now
\begin{align*}
E_n(z)& = B_n\exp(\tilde{B}_n)E_{n+1}(z)\\
& = G_{n-1}^{-1}(z)G_n(z)E_{n+1}(z)\\
& = G_{n-1}^{-1}(z)P(z)E(z).
\end{align*}
Since $G_{n-1}^{-1}(z)P(z)$ is invertible when $|z-z_0|<|z_n-z_0|$, $E(z)$ is also invertible for all $z$ belonging to this disk. Since $|z_n|\to\infty$ as $n\to\infty$, $E(z)$ is invertible for all $z\in\mathbb{C}$.\\
The other factorization of $A(z)$ that appeared in (\ref{AE}) can be proved similarly as above. For clarification, let us mention the first factorization step. Let $Q_1$ be the orthogonal projection on $\ker A(z_1)$, which is of finite rank. Then
$$\left[I-\frac{z-z_0}{z_1-z_0}Q_1\right]^{-1}=\left[I-\frac{z-z_0}{z-z_1}Q_1\right]$$
and since $A(z_1)Q_1=0$,
$$A(z)\left[I-\frac{z-z_0}{z_1-z_0}Q_1\right]^{-1}=\left[\frac{A(z)-A(z_1)}{z-z_1}\right][(z-z_1)-(z-z_0)Q_1]+A(z_1).$$
The remaining steps can be done in an obvious way. This completes the proof.
\end{proof}


\vspace{.2in}
\noindent \textbf{Acknowledgements:} 
The authors sincerely thank the anonymous referees for their insightful and constructive suggestions. In particular, following the recommendation of one referee, Lemma $5.3$ has been improved, and Proposition $6.8$ has been added.
\\
The paper is dedicated to the memory of Professor Harry Dym. We take this opportunity to gratefully acknowledge him for carefully reading an early version of this paper and suggesting several improvements, especially for improving Lemma \ref{Pole} and Theorem \ref{Embedding}.

\vspace{.3cm}

\noindent \textbf{Funding:} 
The research of the first author is supported by the University Grants Commission (UGC) fellowship (Ref. No. DEC18-424729), Govt. of India.\\
The research of the second author is supported by the MATRICS grant of SERB (MTR/2023/001324).

\vspace{.3cm}

\noindent\textbf{Data availability:}\\
No data was used for the research described in the article.

\vspace{.3cm}

 \noindent\textbf{Declarations:}
\vspace{.3cm}

\noindent\textbf{Conflict of interest:}\\
The authors declare that they have no conflict of interest.


\vspace{.5in}


\begin{thebibliography}{25}	
\bibitem{Alpay1} D. Alpay, P. Dewilde, H. Dym, \emph{Lossless inverse scattering and reproducing kernels for upper triangular operators,} in:  Operator Theory: Advances and Applications, Vol. 47, Birkh$\ddot{a}$user, Basel, 1990, pp. 61-135.
\bibitem{Alpay} D. Alpay, A. Dijksma, J. Rovnyak, H. de Snoo, \emph{Schur functions, operator colligations, and reproducing kernel Pontryagin spaces,} Operator Theory: Adv. Appl. 96, Birkh$\ddot{a}$user Verlag, Basel (1997).
	\bibitem{ArD08} D. Arov, H.  Dym, \emph{$J-$Contractive Matrix Valued Functions and Related Topics. } Cambridge University Press, Cambridge, England, 2008.		
	\bibitem{ArD18} D. Arov, H. Dym, \emph{Multivariate prediction, de Branges spaces, and related extension and inverse problems,} Birkh$\ddot{a}$user, Basel (2018).
	\bibitem{ArD20} D. Arov, H. Dym, \emph{ Functional models of operators and their multivalued extensions in Hilbert space.} Integr. Equ. Oper. Theory 92, 39 (2020).
	\bibitem{Bart} H, Bart, T. Ehrhardt, B. Silbermann, \emph{Logarithmic residues of Fredholm operator valued functions and sums of finite rank projections,} In. Linear Operators and Matrices. Operator Theory: Advances and Applications, vol. 130, pp. 83-106, Birkh$\ddot{a}$user Verlag, Basel (2002).	
    \bibitem{Behrndt} J. Behrndt, S. Hassi, H. de Snoo, \emph{Boundary value problems, Weyl functions, and differential operators,} Birkh$\ddot{a}$user/Springer, Cham (2020).
	
	\bibitem{Branges e1} L. de Branges, \emph{Some Hilbert spaces of entire functions,} Proc. Amer. Math. Soc. Vol. 10 (1959) pp. 840-846.
	
	\bibitem{Branges e2} L. de Branges, \emph{Some Hilbert spaces of entire functions,} Trans. Amer. Math. Soc. Vol. 96 (1960) pp. 259-295.
	
	\bibitem{Branges e3} L. de Branges, \emph{Some Hilbert spaces of entire functions,} Bull. Amer. Math. Soc. Vol. 67 (1961) pp. 129-134.
	
\bibitem {Branges 4} L. de Branges, \emph{Hilbert spaces of entire functions,} Prentice-Hall, Inc., Englewood Cliffs, N. J., (1968).
\bibitem{Branges 5} L. de Branges, \emph{The expansion theorem for Hilbert spaces of entire functions,} in : Entire functions and related parts of analysis, in: Proc. Sympos. Pure Math., American Mathematical Society, Providence, Rhode Island, (1968), 79-148.
\bibitem{Rovnyak} L. de Branges, J. Rovnyak, \emph{Canonical models in quantum scattering theory,} in: Perturbation theory and its applications in quantum mechanics, C. Wilcox editor, Wiley, New York (1966).	
\bibitem{Brodskii} M. S. Brodskiĭ, \emph{Triangular and Jordan representations of linear operators,} Translations of mathematical monographs, Vol. 32, American Mathematical Society, Providence, RI, 1971.
\bibitem{Curto} R. E. Curto, I. S. Hwang, W. Y. Lee, \emph{Operator-valued rational functions,} J. Funct. Anal. 283 (2022) 109640.
\bibitem{Derkach1} A. V. Derkach, H. Dym, \emph{ Rigged de Branges–Pontryagin spaces and their application to extensions and embedding,} J. Funct. Anal. 277 (1) 31-110 (2019).
\bibitem{Derkach2} A. V. Derkach, H. Dym, \emph{Functional models for entire symmetric operators in rigged de Branges–Pontryagin spaces,} J. Funct. Anal. 280 (2) 108776 (2021).

\bibitem{Douglas} R. G. Douglas, \emph{On majorization, factorization, and range inclusion of operators on Hilbert space,} Proc. Amer. Math. Soc. 17 (1966), 413–415.
\bibitem{JFA} H. Dym,   S.Sarkar, \emph{Multiplication operators with deficiency indices $(p,p)$ and sampling formulas in reproducing 
	kernel Hilbert spaces of entire vector valued functions,} J. Funct. Anal. {273(2017),3671-3718.}
	\bibitem{Bharti} B. Garg, S. Sarkar, \emph{J-contractive operator valued functions, vector valued de Branges spaces and functional models,} J. Math. Anal. Appl. 549, 129564 (2025).
	\bibitem{Holden} F. Gesztesy, H. Holden, R. Nichols, \emph{On factorizations of analytic operator-valued functions and eigenvalue multiplicity questions.} Integr. Equ. Oper. Theory 82, 61–94 (2015).
	\bibitem{Gohberg1} I. Gohberg, S. Goldberg, M. A. Kaashoek, \emph{Classes of linear operators Vol. $1$,} Operator Theory: Advances and Applications, Birkh$\ddot{a}$user, Basel (1990).
	\bibitem{Gohberg}  I. Gohberg, J. Leiterer, \emph{Holomorphic operator functions of one variable and applications,} Operator Theory: Advances and Applications, vol. 192. Birkh$\ddot{a}$user, Basel (2009).
	\bibitem{Harris} L. A. Harris, \emph{Factorization of operator matrices,} Linear Algebra and its Applications, 225, 37-41 (1995).
	
	
\bibitem{Hörmander} 	L.~V. H\"ormander, \emph{The analysis of linear partial differential operators. III}, corrected reprint of the 1985 original, 
Grundlehren der mathematischen Wissenschaften, 274, Springer, Berlin, 1994.
	
		\bibitem{Woracek} M. Kaltenb$\ddot{a}$ck, H. Woracek, \emph{De Branges spaces of exponential type: general theory of growth,} Acta Sci. Math. (Szeged) 71(1–2), 231–284 (2005).
		\bibitem{Kato} T. Kato, \emph{Perturbation theory for linear operators,} Springer-Verlag, Berlin, Heidelberg, New York (1966).
		\bibitem{Mahapatra} S. Mahapatra, S. Sarkar, \emph{Vector valued de Branges spaces of entire functions based on pairs of Fredholm operator valued functions and functional model,} J. Math. Anal. Appl. 533 (2024) 128010.
		\bibitem{Mahapatra1} S. Mahapatra, S. Sarkar, \emph{Analytic Kramer sampling and quasi Lagrange-type interpolation in vector valued RKHS,} Results Math 79, 230 (2024).
	\bibitem{Paulsen} V. Paulsen, M.  Raghupathi,  \emph{An Introduction to the Theory of Reproducing Kernel Hilbert Spaces} (Cambridge Studies in Advanced Mathematics). Cambridge: Cambridge University Press. (2016) doi:10.1017/CBO9781316219232
	\bibitem{Potapov} V. P. Potapov, \emph{The multiplicative structure of J-contractive matrix functions,} Trudy Moskov. Mat. Obshch. 4 (1955), 125-236; Amer. Math. Soc. Transl. (2) 15 (1960), 131-243.
	\bibitem{Ran} C. L. Prather, A. C. M. Ran, \emph{A Hadamard factorization theorem for entire matrix valued functions,} Operator Theory: Adv. Appl. 19 (1986), 359-372.
	\bibitem{Rosenblum} M. Rosenblum, J. Rovnyak, \emph{Hardy classes and operator theory,} Oxford university press, New York (1985).
	\bibitem{Lay} A. E. Taylor, D. C. Lay, \emph{Introduction to functional analysis,} $2$nd ed., Wiley, New York, $1980$.
	
\end{thebibliography}
\end{document}